\documentclass[10pt]{article}
\usepackage{amsmath, amsthm, setspace, graphicx, subfig}  
\usepackage{latexsym, amsfonts, amssymb}
\usepackage[all,cmtip]{xy} 
\usepackage{subfig}
\usepackage[titles,subfigure]{tocloft}
\usepackage{url}

\newtheorem{thm}{Theorem}
\newtheorem{cor}[thm]{Corollary} 
\newtheorem{lem}[thm]{Lemma} 
\newtheorem{prop}[thm]{Proposition}
\newtheorem{defn}[thm]{Definition}
 
\newtheorem{remarks}[thm]{Remarks}

\def\a{\alpha}
\def\b{\beta} 
\def\g{\gamma} 
\def\d{\delta}
\def\e{\epsilon} 

\def\k{\kappa} 
 
\def\s{\sigma} 
\def\t{\tau} 
\def\o{\omega}

\def\D{\Delta}
\def\L{\Lambda} 
\def\O{\Omega}    

\def\R{\mathbb{R}}

\def\Z{\mathbb{Z}}

\def\p{\partial} 
\def\i{\infty}
\def\cal{\mathcal}
\def\B{\cal{B}}

\def\hB{\hat{\cal{B}}}

\def\A{\cal{A}}

\def\det{\it{det}}

\def\fr{\it{fr}}
\def\ch{\mathcal Ch}
\def\<{\langle}
\def\>{\rangle}

\def\rotateminus{\reflectbox{\rotatebox[origin=c]{155}{\hspace{.6pt}-}}}
\def\Xint#1{\mathchoice
{\XXint\displaystyle\textstyle{#1}}%
{\XXint\textstyle\scriptstyle{#1}}%
{\XXint\scriptstyle\scriptscriptstyle{#1}}%
{\XXint\scriptscriptstyle\scriptscriptstyle{#1}}%
\!\int}
\def\XXint#1#2#3{{\setbox0=\hbox{$#1{#2#3}{\int}$}
\vcenter{\hbox{$#2#3$}}\kern-.5\wd0}}

\def\cint{\Xint \rotateminus }

\numberwithin{thm}{subsection} 
\numberwithin{equation}{section} 

\begin{document} 
		\title{Soap film solutions to Plateau's problem}
	\author{J. Harrison\footnote{The author was partially supported by the Miller Institute for Basic Research in Science and the Foundational Questions in Physics Institute. An early version of this paper was first posted on the arxiv in February, 2010} \\Department of Mathematics \\University of California, Berkeley}

	\maketitle
	\begin{abstract}     
	  Plateau's problem is to show the existence of an area minimizing surface with a given boundary, a problem posed by Lagrange in 1760.  Experiments conducted by Plateau showed that an area minimizing surface can be obtained in the form of a film of oil stretched on a wire frame, and the problem came to be called Plateau's problem.  Special cases have been solved by Douglas, Rado, Besicovitch, Federer and Fleming, and others.   Federer and Fleming used the chain complex of integral currents with its continuous boundary operator to solve Plateau's problem for orientable, embedded surfaces.  But integral currents cannot represent surfaces such as the M\"obius strip or surfaces with triple junctions. In the class of varifolds, there are no existence theorems for a general Plateau problem because of a lack of a boundary operator.  We use the chain complex of differential chains with its continuous boundary operator to solve a general version of Plateau's problem.  We find the first solution which minimizes area taken from a collection of surfaces that includes all previous special cases, as well as all smoothly immersed surfaces of any genus type,  both orientable and nonorientable, and surfaces with multiple junctions.  
	Our result holds for all dimensions and codimensions in \( \R^n \).          
	\end{abstract}
 
 	\synctex=1
 		\pagenumbering{roman}
 			
 			\pagenumbering{arabic}
 		
 \section{Introduction} 
\label{sec:introduction}  
 
Plateau's problem asks whether there exists a  surface with minimal area  that spans a prescribed smoothly embedded closed curve \( \g \).  The solution depends on the definitions of  ``surface'', ``area'',    and ``span''.      Given a collection \( {\cal C} \) of surfaces, there is a natural sequence of questions:  
\begin{enumerate}
	\item Does there exist \( S \in {\cal C} \) that spans \( \g \)?
	\item Is the infimum \( m \) of areas of surfaces spanning \( \g \) nonzero?
	\item Does there exist a surface \( S_0 \) spanning \( \g \) with area \( m \)?
	\item What is the structure of \( S_0 \) away from \( \g \)?
	\item What is the structure of \( S_0 \) near \( \g \)?
\end{enumerate} 
                   
  Historically,  affirmative answers to  (1) and (3)   have been celebrated as a \emph{solution to Plateau's problem}, leaving  questions of regularity (4) and (5) for further work  over a period of time.    The general problem of solving (1)-(3) for a collection \( {\cal C} \) containing all known soap\footnote{Plateau's original experiments \cite{plateauoriginal} were with oil which does not form films across wire frames as easily as soap/glycerin solutions.} films arising in nature has been an open problem for 250 years.     Identifying a  new chain complex of topological vector spaces with a  rich algebra of bounded operators gave us a new approach to tackle the problem.  We present the first existence theorem for the general problem of Plateau in a setting where the boundary operator is continuous.

The first solutions to Plateau's problem by  Douglas \cite{douglas}, for which he won the first Field's medal, were found by defining surfaces as parametrized images of a disk,  and thus did not permit  nonorientable surfaces   or triple junctions.   Douglas used the integral of the Jacobian of the parametrizing map  to define area.   Figure \ref{AllThree.jpg} shows that the ``classical solutions'' of Douglas can have transverse self-intersections which are never seen in soap films.   Osserman  \cite{osserman}, Alt \cite{alt},  and  Gulliver \cite{gulliver} proved any classical solution of Douglas must be an immersion of a disk.   The solutions of Federer and Fleming \cite{federerfleming} using surfaces defined as integral currents are necessarily orientable.  They define area using the  mass norm of a current. They were awarded the Steele prize  ``for their pioneering work in Normal and Integral currents''  \cite{federerfleming}.    Fleming \cite{flemingregular} proved such solutions are smoothly embedded away from \( \g \), and regularity near a smooth boundary (5) was later established in \cite{hardtsimon}.    Plateau's problem remained an open problem because none of the  solutions given by Douglas or Federer and Fleming permits the M\"obius strip  as the solution \( (b) \) for the curve in Figure \ref{MStrip.jpg}.  Instead, their methods produce solution \( (a) \), an oriented embedded disk.   
	
 \begin{figure}[htbp]
 	\centering
 		\includegraphics[height=3in]{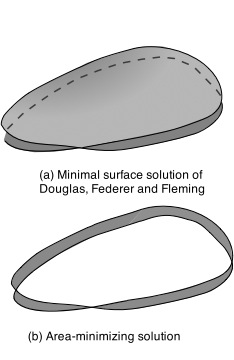}
 	  \caption{The M\"obius strip}
 	\label{MStrip.jpg}
 \end{figure}

   Reifenberg \cite{reifenberg} used point sets to define his surfaces and Hausdorff measure for area, but did not have a boundary operator.  His surfaces could model nonorientable examples, but not those with triple junctions.  Fleming and Ziemer's flat chains (mod 2) \cite{fleming, ziemer} contain M\"obius strips, but no surfaces with triple junctions.   Fleming's flat chains (mod 3) permit triple junctions but not M\"obius strips.
 
    Almgren's integral varifolds \cite{varifolds}   provide models for all soap films. A $2$-varifold is defined as a Radon measure on the product of \( \R^3 \) with the Grassmannian of $2$-planes through the origin of \( \R^3 \).   He proved  a   compactness theorem for integral varifolds with bounds on areas, first variations and supports.  For a time,  there was excited optimism about Almgren's methods.  Ziemer's Bulletin review \cite{ziemerplateau} referred to varifolds as  ``a new and promising approach to the old and formidable Plateau's problem.'' 
However, the lack of a boundary operator on varifolds  (see, for example, \cite{morgan}, \S 11.2), has made the proof of existence of an area minimizer given by a compactness theorem for varifolds  elusive.    
  The problem remained open with most mathematicians not realizing it. 
 
In this paper,   we    answer (1)-(3) in the affirmative. Our solutions are geometrically meaningful,  not just weak solutions, since all differential $k$-chains are approximated by ``Dirac $k$-chains'' which we define as formal sums \( \sum_{i=1}^{m} (p_i; \a_i) \) where \( p_i \in M \) and \( \a_i \in \L_k(T_{p_i}(M)) \).  Dirac chains have a natural and simple constructive geometric description (see \S \ref{sub:representations_of_open_sets}).    Our methods extend to a number of other variational problems and  to $k$-dimensional  cycles in Riemannian $n$-manifolds \( M \) for all \(0 \le k \le n-1 \).   

We roughly state our main theorem for smoothly embedded closed curves in \( \R^3 \) before defining all of the terms.   
\begin{thm}\label{thm:Plat}
Given a smoothly embedded closed curve \( \g:S^1 \to   \R^3 \), there exists a surface \( S_0 \) spanning \( \g = \g(S^1) \) with minimal area where \( S_0 \) is an element of a certain topological vector space   which  includes representatives of   all types of observed soap films as well as all smoothly immersed surfaces of all genus types, orientable or nonorientable, including those with possibly multiple junctions.
\end{thm}

See Theorem \ref{pro:m} for a precise statement.  Our compactness Theorem \ref{pro:com} leading to Theorem \ref{thm:Plat}  is the first compactness theorem in an infinite dimensional space taking into account all  known soap films and  smoothly immersed surfaces, and  for which the boundary operator is well-defined, continuous, and maps the solution's ``representative''  to a representative \( \widetilde{\g} \) of the prescribed curve \( \g \).    Figure \ref{EmptyLoop.jpg} shows that \( \g \) does not have to be a closed curve\footnote{This example was used by Almgren to defend the lack of a boundary operator for varifolds as natural}.   However, \( \p \widetilde{\g} = 0 \), even for the example in Figure \ref{EmptyLoop.jpg}, without the need for ``hidden wires''.     

Figures \ref{MStrip.jpg} and \ref{AllThree.jpg} demonstrate the differences between solutions found by Douglas, Federer, Fleming, and the author for the  M\"obius strip and a simple modification of it.        
\begin{figure}[htbp]
	\centering
		\includegraphics[height=3in]{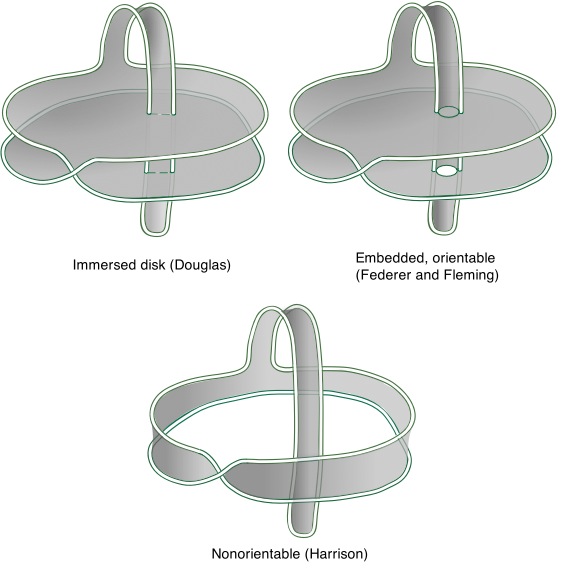}
  \caption{Three different solutions from three different viewpoints}
	\label{AllThree.jpg}
\end{figure}

In \cite{soap, plateau} the author found a solution to Plateau's soap film problem,   assuming a   bound on the total length of triple junctions,  a condition now discarded using the new methods of in \cite{OC}.    

The author thanks Morris Hirsch and Harrison Pugh for their helpful comments, questions, and insights.  Steven Krantz has supported research of numerous mathematicians working on aspects of Plateau's problem through his exceptional editorial work for the Journal of Geometric Analysis.  This paper would surely not exist had it not been for his early interest and encouragement.  She wishes to acknowledge Frank Morgan  \cite{morgan}  for his course on geometric measure at the Mathematical Sciences Research Institute in 2001 where he informed the audience that the general problem of Plateau was still open.  

 \section{Differential chains of type $ B $}          
\label{sub:preliminaries} 
   	  
   This work relies on methods of calculus presented in \cite{OC}.  In this preliminary section we recount the definition of the bigraded chain complex of topological vector spaces \( \hB_k^r(U) \) and those operators  from \cite{OC}  which we use in this paper.  
 \subsection{Dirac chains} 
 \label{Mackey}
 
For \( U \) open in \( \R^n \),  let \( {\cal A}_k =  {\cal A}_k(U) \) be the free vector space of \emph{Dirac $k$-chains in \( U \)}, i.e., finitely supported functions \( U \to \L_k(\R^n) \), expressed in the formal sum notation \( \sum (p_i; \a_i) \) where \( p_i \in U \) and \( \a_i \in \L_k(\R^n) \). (We use the standard convention of formal sums in which the only relations permitted are when the base points are the same.  For example, \(  (p;\a) + (p;\b) = (p;\a+\b) \) and \( 2 (p;\a) = (p;2\a) \).  ) We call \( (p;\a) \) a   $k$-\emph{element in \( U \)} if \( \a \in \L_k(\R^n) \) and \( p \in U \).  If \( \a \) is simple, then \( (p;\a) \) a \emph{simple} $k$-element in \( U \).

  \subsection{Mass norm} \label{ssub:differential_chains}
     An inner product \( \<\cdot,\cdot\> \) on \( \R^n \) determines the mass norm on \( \L_k(U) \) as
		follows: Let \( \<u_1 \wedge \cdots \wedge u_k, v_1 \wedge \cdots \wedge v_k \> = \det(\<u_i,v_j\>) \).  The	\emph{mass} of a simple $k$-vector \( \a = v_1 \wedge \cdots \wedge v_k \) is defined by \(
		\|\a\| := \sqrt{\<\a,\a\>} \). The mass of a $k$-vector \( \a \) is \( \|\a\| := \inf\left\{\sum_{j=1}^{N} \|(\a_i)\| : \a_i \mbox{ are simple, } \a =
		\sum_{j=1}^N \a_i \right\}. \) Define the \emph{mass} of a $k$-element \( (p;\a) \) by  \(\|(p;\a)\|_{B^0} := \|\a\|  \). Mass is a norm on the subspace of Dirac $k$-chains
		supported in \( p \), since that subspace is isomorphic to the exterior algebra \(
		\L(\R^n) = \oplus_{k=0}^n \L_k(\R^n) \) for which mass is a norm (see \cite{federer}, p 38-39). The \emph{mass} of a Dirac $k$-chain \( A =  \sum_{i=1}^{m} (p_i; \a_i) \in {\cal A}_k(U) \)   is given by  
	   \[  \|A\|_{B^0} := \sum_{i=1}^{m} \|(p_i; \a_i)\|_{B^0}. \] If a different inner product is chosen, the resulting masses of Dirac	chains are topologically equivalent.    It is straightforward to show that \( \|\cdot\|_{B^0} \) is a norm on \( {\cal A}_k(U) \).  
	
			\subsection{Difference chains and the  $ B^r $ norm} \label{sub:difference_chains}

		 Given a $k$-element \( (p;\a) \) with \( p \in U \) and \( u \in \R^n \), let \( T_u(p;\a) := (p+u;\a) \)
		be translation through \( u \), and \( \D_u(p;\a) := (T_u -I)(p; \a) \). Let \( S^j =
		S^j(\R^n) \) be the \( j \)-th \emph{symmetric power} of the symmetric algebra \( S(\R^n)
		\). Denote the symmetric product in the symmetric algebra \( S(\R^n) \) by \( \circ \). Let \( \s = \s^j = u_1 \circ \dots
		\circ u_j \in S^j \) with \( u_i \in \R^n, i = 1, \dots, j \). Recursively define $\D_{u
		\circ \s^j}(p; \a) := (T_u - I)(\D_{\s^j}(p; \a))$. Let \( \|\s\| := \|u_1\| \cdots
		\|u_j\| \) and \( |\D_{\s^j}(p; \a)|_{B^j} := \|\s\| \|\a\|. \)  Let \( \D_{\s^0} (p;\a) := (p;\a) \), to keep the notation consistent.   We say \( \D_{\s^j} (p;\a) \) is \emph{inside} \( U \) if the convex hull of \( | \D_{\s^j} (p;\a)| \) is a subset of \( U \).   

\begin{defn} \label{def:norms} For \( A \in {\cal A}_k(U) \) and \( r\ge 0 \), define the seminorm \[   \|A\|_{B^{r,U}} := \inf \left \{ \sum_{j=0}^r \sum_{i=1}^{m_j} \|\s_{j_i}\|\|\a_{j_i}\|: A = \sum_{j=0}^r \sum_{i=1}^{m_j} \D_{\s_{j_i}^j}(p_{j_i};\a_{j_i}) \mbox{ where } \D_{\s_{j_i}^j}(p_{j_i};\a_{j_i})  \mbox{ is inside } U \right\}. \] 
\end{defn}   For simplicity, we often write \( \|A\|_{B^r} =  \|A\|_{B^{r,U}} \) if \( U \) is  understood.   It is easy to see that the \( B^r \) norms on Dirac chains are decreasing as \( r   \) increases.

    It is shown in \cite{OC} (Theorems 2.6.1 and 6.1.2)  that \( \|\cdot\|_{B^r} \) is a norm on the free space of Dirac $k$-chains \( {\cal A}_k  \) called the \emph{\( B^r \) norm}.  
 		Let \( \hB_k^r = \hB_k^r(U) \) be the Banach space obtained upon completion of \( {\cal A}_k(U) \) with the \( B^r \) norm.  Elements of \( \hB_k^r(U), 0 \le r \le \i \),  are called \emph{differential   $k$-chains of class \( B^r \) in \( U \)}. The natural inclusions \(  \hB_k^r(U_1) \hookrightarrow  \hB_k^r(U_2) \) are continuous for all open \( U_1 \subset U_2 \subset \R^n \).  In \cite{OC} we also study the inductive limit  \( \hB_k =  \hB_k(U)    := \varinjlim \hB_k^r(U) \)   as \( r \to \i \), endowed with the inductive limit topology, obtaining a \( DF \)-space.  However, for this paper it suffices to work within a subcomplex of the simpler bigraded chain complex of Banach spaces \( \hB_k^r \).

Let \( {\cal B}_k^0(U) \) be the Banach space of bounded and measurable $k$-forms, \( {\cal B}_k^1(U) \) the Banach space of bounded Lipschitz $k$-forms, and for  each  \( r > 1 \),   let \( {\cal B}_k^r(U) \) be the Banach space of differential $k$-forms, each with uniform bounds on each of the \( s \)-th order directional derivatives for \( 0 \le s \le r-1 \) and the \( (r-1) \)-st derivatives satisfy a bounded Lipschitz condition.  Denote the resulting norm by \( \|\o\|_{B^r} = \sup_{|\ell| \le r-1}\{\| \o\|_{\sup}, |D^{\ell} \o|_{Lip}  \}  \).  We always denote differential forms by lower case Greek letters such as \( \o, \eta \) and differential chains by upper case Roman letters such as \( J, K, A \), so there is no confusion when we write \( \|\o\|_{B^r} \) or \( \|J\|_{B^r} \).  Elements of \( {\cal B}_k^r = {\cal B}_k^r(U) \) are called \emph{differential   $k$-forms of class \( B^r \) in \( U \)}.  

\begin{thm}[Isomorphism Theorem]\label{thm:isomorphism}
  \(  \hB_k^r(U)' \cong {\cal B}_k^r(U)  \)  and the integral pairing \( \cint: \hB_k^r(U) \times {\cal B}_k^r(U) \to \R \) where \( (J,\o) \mapsto \o(J) \) is bilinear, nondegenerate, and separately continuous.  
\end{thm}  (See \cite{OC} Theorem 6.2.3)

Denote  \[ \cint_J \o := \o(J) \] for all \( J \in \hB_k^r(U) \) and \( \o \in  {\cal B}_k^r(U) \).                                                               	 
      
\begin{defn}  \label{def:support}
If \( J \in \hB_k^r(U) \) is nonzero then its \emph{\textbf{support}} \( |J| \)  is the smallest closed subset \( E \subset U \)   such that  \( \cint_J \o = 0 \) for all smooth \( \o \) with compact support disjoint from \( E \).  Support of a nonzero differential chain  is a uniquely determined nonempty set (see \cite{OC}  Theorems 5.0.6 and 6.4.5).                           
	
\end{defn}

\subsection{Pushforward} 
\label{sub:operators1}  
 Suppose \( U_1 \subset \R^n\) and \( U_2 \subset \R^m \) are open  and \( F:U_1 \to
U_2 \) is a differentiable map. For \( p \in U_1 \) and \( v_1 \wedge \cdots \wedge v_k \in \L_k(\R^n) \), define \emph{linear pushforward} \( F_{p*}(v_1 \wedge \cdots \wedge
v_k) := DF_p(v_1) \wedge \cdots \wedge DF_p (v_k) \) where \( DF_p \) is the total
derivative of \( F \) at \( p \).     Define  \( F_*(p; \a) := (F(p), F_{p*}\a) \) for all simple $k$-elements \( (p;\a) \) and extend  to a linear map \( F_*:{\cal A}_k(U_1) \to {\cal A}_k(U_2) \) called \emph{pushforward}.  Define \( F^* \o := \o F_* \) for exterior $k$-forms \( \o \in {\cal A}_k(U)^* \).  Then \( F^* \) is the classical  pullback  \( F^*:{\cal A}_k(U_2)^* \to {\cal A}_k(U_1)^* \).    

\begin{defn}\label{def:Mr} \mbox{} \\
  Let \( {\cal M}^r(U, \R^m) \) be the vector space of Lipschitz maps \( F:U \to \R^m \) each of whose  directional derivatives \( L_{e_j} F_i \) of its coordinate functions \( F_i \)  are of class \( B^{r-1} \), for \( r \ge 1 \). Define the seminorm  \(  \rho_r(F) := \max_{i,j} \{\|L_{e_j} F_i\|_{B^{r-1, U}}\} \).    Let \( {\cal M}^r(U_1, U_2)  :=  \{F \in {\cal M}^r(U_1, \R^m) : F(U_1) \subset U_2 \subset \R^m\} \).   
	 \end{defn}      
A map \( F \in {\cal M}^1(U, \R^m) \) may not be bounded, but its directional derivatives must be.  An important example is the identity map \( x \mapsto x \)  which is an element of \(  {\cal M}^1(\R^n,\R^n) \).     
	     
\begin{thm}\label{thm:A}
 If \( F \in {\cal M}^r(U_1, U_2) \), then  \( F_* \) satisfies 
 \[
	\|F_*(A)\|_{B^{r,U_2}} \le  n2^r \max\{1, \rho_r(F)\} \|A\|_{B^{r,U_1}} \] for all
	\( A \in {\cal A}_0(U_1) \) and  \( r \ge 1 \).  It follows that  
 \( F_*: \hB_k^r(U_1) \to \hB_k^r(U_2) \) and  \( F^*:{\cal B}_k^r(U_2) \to {\cal B}_k^r(U_1) \)  are continuous bigraded operators with \( \cint_{F_* J} \o = \cint_J F^* \o \)  for all \( J \in \hB_k^r(U_1) \) and \( \o \in {\cal B}_k^r(U_2) \). Furthermore, 	 \( \p \circ F_* = F_* \circ \p \).
\end{thm}  

(See Theorem 6.5.6 in \cite{OC}.)  
 
	\subsubsection{Algebraic chains} 
	\label{sub:representations_of_open_sets}\mbox{}\\ 
\begin{thm}\label{thm:B}[Representatives of $k$-cells]
	Each oriented affine $k$-cell  \( \s \) in  \( U \) naturally corresponds to a unique  differential	$k$-chain  \( \widetilde{\s} \in \hB_k^1(U) \) 
	in the sense of integration of differential forms. That is, the  Riemann integral and the 
	differential chain integral coincide for all Lipschitz k-forms  \( \o \): \[  \int_\s \o = \cint_{\widetilde{\s}} \o. \]  
\end{thm}

The proof of this theorem may be found in \cite{OC} (Theorem 2.10.2).	
It follows that each bounded open set \( U \) in \( \R^n \) with the standard orientation of \( \R^n \) is uniquely represented by an $n$-chain \( \widetilde{U} \in \hB_n1(U) \). We define \emph{polyhedral chains} \( \sum_{i=1}^{s} a_i \widetilde{\s_i} \), \( a_i \in \R \) and \( \s_i \) an oriented affine $k$-cell.  This coincides with the classical definition as found in \cite{whitney}.  Polyhedral chains are dense in \( \hat{B}_k^r(U) \) (see \cite{OC} Theorem 2.10.5).

 If \( \s \) is an oriented affine $k$-cell in \( U \) and \( F \in M^1(U,W) \),  then \( F_* \widetilde{\s}  \in \hB_k^1(W)  \), and is called an \emph{algebraic $k$-cell}.   We remark that an algebraic $k$-cell \( F_* \widetilde{\s} \) is not the same as a singular $k$-cell \( F\s \) from algebraic topology. For example, if \( F(x) = x^2 \) and \( \s = (-1, 1) \), then the algebraic $1$-cell \( F_* \widetilde{\s} = 0 \), but the singular $1$-cell \( F \s \ne 0 \).  
	
 If \( (p;\a) \) is a simple $k$-element, we may write  \( (p;\a) = \lim_{i \to \i} 2^{ki} Q_i(p) \) where \( Q_1(p) \) is an oriented affine  $k$-cell in the $k$-direction of \( \a \) containing \( p \) with unit diameter, unit $k$-volume, and \( Q_i(p) \) is a homothetic replica of \( Q_1(p) \), containing \( p \), and with diameter \( 2^{-i} \) (see \cite{OC} Lemma 2.10.4).   This gives us the promised geometric interpretation of the simple $k$-element \( (p;\a) \) as a $k$-dimensional point mass, a  limit of shrinking renormalized oriented affine $k$-cells.   (We can also use any sequence of limiting chains as long as their supports tend to \( p \), their $k$-directions are the same and the masses tend to \( \|\a\|_0 \).  There is nothing special about \( k \)-cells here, except for computational convenience. )

	\subsection{Vector fields} 
	\label{sub:vector_fields}                              
Let \( {\cal V}^r(U) \) be the vector space of vector fields \( X \) on \( U \) whose local coordinate functions \( \phi_i \) are of class \( B^r \). In particular, if \( r = 0 \), then the time-\( t \) map of the flow of \( X \) is Lipschitz.   For \( X \in {\cal V}^r(U) \)  define \( \|X\|_{B^r} = \max\{\|\phi_i\|_{B^r}\} \).  Then \( \|\cdot\|_{B^r} \) is a norm on \( {\cal V}^r(U) \).
We say that \( X \) is of \emph{class \( B^r \)} if \( X \in {\cal V}^r(U) \).   	

 \section{Operators} 
 \label{sec:the_quantum_differential_complex}
 
 \subsection{Extrusion} 
\label{sub:extrusion}

  Let \( X \in {\cal V}^r(U) \).  Define the graded operator \emph{extrusion} \( E_X:{\cal A}_k(U) \to {\cal A}_{k+1}(U) \) by \( E_X (p;\a) := (p; X(p) \wedge \a) \) for all \( p \in U \) and \( \a \in \L_k(\R^n) \).     Then \( i_X \o := \o E_X  \) is the classical  interior product    \( i_X:{\cal A}_{k+1}(U)^* \to {\cal A}_k(U)^* \).    
	
\begin{thm}\label{thm:C}
 If \( X \in {\cal V}^r(U) \) and \( A \in {\cal A}_k(U) \),  then   \[  \|E_X(A)\|_{B^r} \le    n^2r\|X\|_{B^r}\|A\|_{B^r}.  \] 
\end{thm}	 
(For a proof see \cite{OC}, Theorem 8.2.2.)     
	
	Therefore,  \( E_X \) extends to a continuous operator   \( E_X: \hB_k^r(U) \to \hB_{k+1}^r(U)   \).    It follows from the isomorphism theorem  \ref{thm:isomorphism} that    
\begin{equation}\label{EXthm}
  \cint_{E_X J} \o = \cint_J i_X \o.   
\end{equation}    
  	
	\subsection{Retraction} 
	\label{sub:contraction}         For \( \a = v_1 \wedge \cdots \wedge v_k \in \L_k(\R^n) \), let \( \hat{\a}_i := v_1 \wedge
	\cdots \hat{v_i} \cdots \wedge  v_1 \in \L_{k-1}(\R^n) \). 
	 For \( X\in {\cal V}^r(U) \) define the graded operator  \emph{retraction} \( E_X^\dagger:  {\cal A}_k(U)  \to {\cal A}_{k-1}(U) \)  by \( (p; \a)  \mapsto \sum_{i=1}^k (-1)^{i+1} \<X(p),v_i\>  (p;  \hat{\a}_i),  
	\)  for \( p \in U \).  A straightforward calculation shows this to be the adjoint of wedge product with \( X(p) \) at a point \( p \), and thus is well-defined.       The dual operator on forms is wedge product with the \( 1 \)-form  \( X^\flat \) representing the vector field \( X \) via the inner product with \( X^\flat \wedge \cdot: {\cal A}_{k-1}(U)^* \to {\cal A}_k(U)^*  \).   
\begin{thm}\label{thm:Cdagger}
 If \( X \) is a vector field on \( U \) of class \( B^r \) and \( J \in  \hB_k^r(U) \),  then \[  \|E_X^\dagger(J)\|_{B^r} \le    k {n \choose k} \|X\|_{B^r}\|J\|_{B^r}  \]   or all \( r \ge  1 \).
\end{thm}   
	        	  It follows that     \( E_X^\dagger: \hB_k^r(U) \to \hB_{k-1}^r(U)   \) and \(   X^\flat \wedge \cdot: {\cal B}_{k-1}^r(U) \to {\cal B}_k^r(U)  \) are continuous graded operators with 
	\begin{equation}\label{eq:EXD}
	  \cint_{E_X^\dagger J} \o = \cint_J  X^\flat \wedge \o 
	\end{equation}
	for all   \( J \in \hB_k^{r+1}(U) \) and \( \o \in {\cal B}_k^r(U) \)
	  (see \cite{OC} Theorems 8.3.3 and  8.3.5)     

 \subsection{Boundary} 
 \label{sub:boundary}
 
There are several equivalent ways to define the boundary operator  \( \p:\hB_k^r(U) \to \hB_{k-1}^{r+1}(U) \) for \( r \ge 0 \).     We have found it very useful to define boundary on Dirac chains directly.   For \( v \in \R^n \), and a simple $k$-element \( (p;\a) \) with \( p \in U \), let \( P_v (p;\a) := \lim_{t \to 0} (p+tv;\a/t) - (p;\a/t) \).  It is perhaps surprising that this limit is nonzero if \( \a \ne 0 \). It is shown in \cite{OC} (Lemma 3.3.1) that this limit exists as a well-defined element of \( \hB_k^2(U) \).  We may then linearly extend \( P_v:{\cal A}_k(U) \to \hB_k^1(U) \).  Moreover, \( \|P_v (A)\|_{B^{r+1}} \le \|v\|\|A\|_{B^r} \) for all \( A \in {\cal A}_k(U) \)   (see \cite{OC} Lemma 3.3.2). For an orthonormal basis \( \{e_i\} \) of \( \R^n \), set \( \p := \sum P_{e_i} E_{e_i}^\dagger  \).  Since \( P_{e_i} \) and \( E_{e_i} \) are continuous,  \( \p \) is a well-defined continuous operator \( \p:\hB_k^r(U) \to \hB_{k-1}^{r+1}(U) \) that restricts to the classical boundary operator on polyhedral $k$-chains independent of choice of \( \{e_i\} \) (see  \cite{OC} (Corollary 3.5.2 and Lemma 3.5.5)).  

\begin{thm}\label{thm:D}[General Stokes' Theorem]
 The bigraded operator boundary  \( \p:\hB_k^r(U)  \to  \hB_{k-1}^{r+1}(U) \)  is continuous  with \( \p \circ \p = 0 \),  and  \( \|\p J\|_{B^{r+1}} \le k n\|J\|_{B^r} \)   for all \( J \in \hB_k^r \) and \( r \ge 0 \).   Furthermore,  if \( \o \in {\cal B}_{k-1}^r(U) \) is a differential form and \( J \in \hB_k^{r-1}(U) \) is a differential chain, then   \[ \cint_{\p J} \o = \cint_J d \o. \]
\end{thm} 
   	 The proof of this may be found in \cite{OC} (Theorems  3.5.2, 3.5.4, 6.7.1, and 6.7.2). 

	We say a differential $k$-chain \( J \in \hB_k(U)  \) is a \emph{differential $k$-cycle in \( U \)} if $\p J = 0.$ 

\subsection{Prederivative} 
\label{sub:prederivative}

\begin{defn}\label{def:preder}
	Suppose \( X \in {\cal V}^r(U) \).  Define the  linear map \emph{\textbf{prederivative}} \( P_X:  \hB_k^r(U) \to \hat{\cal
B}_k^{r+1}(U) \) by \[  P_X  := E_X \p  + \p E_X. \] 
\end{defn} This agrees with the previous definition of \( P_v \) for \( v \in \R^n \) in \( \S\ref{sub:boundary} \) since \\  
 \(  E_v \p + \p E_v = \sum_i  P_{e_i}( E_v E_{e_i}^\dagger - E_{e_i}^\dagger E_v) = \sum_i P_{e_i} \<v,e_i\> I = P_v.  \) 
 
 It follows from Theorems  \ref{thm:C} and  \ref{thm:D} that both \( E_X \) and \( \p \) are
continuous.  Therefore,  \( P_X \) is continuous.  
Its dual operator \( L_X \) is the classically defined Lie derivative    since \( L_X = i_X d + d i_X \) by Theorems \ref{thm:C} and  \ref{thm:D}, and this uniquely determines \( L_X \).   
It follows that \( P_X: \hB_k^r(U) \to \hB_k^{r+1}(U)  \) and  Lie derivative \( L_X: {\cal B}_k^r(U) \to {\cal B}_k^{r-1}(U) \) are continuous bigraded linear operators. By the isomorphism theorem \ref{thm:isomorphism}, this implies the duality relation:   
\begin{equation}\label{eq:PXL}
\cint_{P_X J} \o = \cint_J L_X \o	
\end{equation}     
for all   \( J \in \hB_k^{r-1}(U) \) and \( \o \in {\cal B}_k^r(U) \). Furthermore, \( P_X \p = \p P_X \) since \( P_X = \p E_X + E_X \p \) implies \( \p E_X = \p E_X \p = P_X \p \).      We remark that \( P_u \circ P_v = P_v \circ P_u \) for fixed \( u, v \in \R^n \), but for non-constant vector fields \( X, Y \), the operators \( P_X, P_Y \) do not necessarily commute\footnote{The universal enveloping algebra can be used instead of the symmetric algebra.}.   

\begin{thm}\label{thm:E}
	If \( X \in {\cal V}^r(U) \),  then  prederivative  
	\begin{align*} P_X: \hB_k^r(U) \to \hB_k^{r+1}(U)  
	\end{align*} satisfies
	\[ \|P_X(J)\|_{B^{r+1}} \le      2kn^3 2^r \|X\|_{B^{r}} \|J\|_{B^r} \] for all \( J \in \hB_k^r(U) \) and \( 0 \le r < \i \).    
\end{thm}  

(See Theorem 8.4.2 of \cite{OC})    

Prederivative gives us a way to ``geometrically differentiate'' a differential chain in the infinitesimal directions determined by a vector field,  even when the support of the differential chain is highly nonsmooth, and without using any functions or forms.

\begin{thm}\label{thm:F}	
 If \( X \in {\cal V}^{r+1}(U) \)  and \( J \in \hat{B}_k^r(U), r \ge 0 \), has compact support, then \( P_X J \in \hat{B}_k^{r+1}(U) \) with \[ 	P_X J = \lim_{t \to 0}( \phi_{t*} J/t - J/t) \] where \( \phi_t \) is the time-\( t \) map of the flow of \( X \).
\end{thm}
     The proof of this may be found in  \cite{OC} (Theorem 8.4.4). 
		    
	\subsubsection{The chainlet complex} 
	\label{sub:a_topology_on_(_)}
	If \( \t = u_s \circ \cdots \circ u_1 \in S^s(\R^n) \), the \( s \)-th order symmetric power of the symmetric algebra, let \( P_\t:= P_{u_s} \circ \cdots \circ P_{u_1} \).
	\begin{defn}\label{differential chains}  	For \( U \) open in \( \R^n \), let \( \A_k^s(U) := \{\sum P_{\t_i} (p_i;  \a_i) \,|\, p_i \in U, \t_i  \in S^s(\R^n),  \mbox{ and } \a_i \in \L_k(\R^n) \} \).  Elements of \( \A_k^s(U) \) are called \emph{\textbf{Dirac $k$-chains of dipole order \( s \) in \( U \)}}.
	 Let \[ \ch_k^s(U) :=
	\overline{({\cal A}_k^s(U), \|\cdot\|_{B^{s+1}})}. \]   Elements of \( \ch_k^s(U) \) are called \emph{\textbf{$k$-chainlets of dipole order \( s \)}}.  For \( s \ge 1 \), the space  \( \ch_k^s(U) \) is a strict topological subspace of the Banach space \( \hB_k^{s+1}(U) \), while \( \ch_k^0(U) = \hB_k^1(U) \).
	\end{defn}  In this paper, we only need \( 0 \le s \le 1 \) and \( n-1 \le k \le n \), but include the entire differential chain complex and its chainlet subcomplex here for completion.  
	
	It is not hard to see that the primitive operators \( E_V, E_V^\dagger \) and \( P_V \), as well as the operators pushforward and multiplication by a function,  are continuous and closed on the direct sum \(\oplus_k \oplus_s \ch_k^s(U) \).   It follows that  boundary is continuous, and thus \( \ch_k^s(U) \) is a bigraded topological chain complex and  is a proper subcomplex of the differential chain complex \( \hB_k^{s+1}(U) \) with the induced topology.
	\[
		\xymatrix{\hB_n^{s+1} \ar[r]^\p &\hB_{n-1}^{s+2} \ar[r]^\p &\cdots \ar[r]^\p &\hB_1^{n+s} \ar[r]^\p &\hB_0^{n+s+1} \\ \ch_n^s \ar[r]^\p \ar[u] &\ch_{n-1}^{s+1} \ar[r]^\p \ar[u] &\cdots \ar[r]^\p  &\ch_1^{n+s-1} \ar[r]^\p \ar[u] &\ch_0^{n+s}\ar[u] \\  \A_n^s \ar[r]^\p \ar[u] &\A_{n-1}^{s+1}  \ar[u] \ar[r]^\p   &\cdots \ar[r]^\p  &\A_1^{n+s-1} \ar[r]^\p \ar[u] &\A_0^{n+s}\ar[u] }   
	\]
  In \S\ref{sub:surfaces} we work almost entirely in the chainlet complex.

	\section{Integral monopole and dipole chains}
 
\begin{defn}
  Suppose \( X \) is a Lipschitz vector field defined in a neighborhood of a smoothly embedded $(n-1)$-cell \( \t \subset U \subset \R^n \).  For \( p \in \t \), let \( \a(p) \) be the unit $(n-1)$-vector tangent to \( \t \) at \( p \).   Assume that   \( X(p) \wedge \a(p) \) has unit mass and is positively oriented.    We call \( E_X  \widetilde{\t} \in \ch_n^0(U) \) an 	\emph{\textbf{integral monopole $n$-cell}} and \( P_X  \widetilde{\t} \in \ch_{n-1}^1(U) \)  an	\emph{\textbf{	integral dipole $(n-1)$-cell }} in \( U \).    Finite sums \(  \sum_{i=1}^m E_{X_i} \widetilde{\t_i} \) and \(  \sum_{i=1}^m P_{X_i} \widetilde{\t_i} \) are   \emph{\textbf{integral monopole and dipole chains }} in \( U \), respectively, and form convex cones \( {\cal I}_n^0(U)  \subset \ch_n^0(U) \) and \( {\cal I}_{n-1}^1(U) \subset \ch_{n-1}^1(U) \), respectively.  	
\end{defn}

The completion \( \overline{{\cal I}_{n-1}^1(U)} \)  uses the \( B^2 \) norm, and \( \overline{{\cal I}_n^0(U)} \) uses the \( B^1 \)-norm. 

If \( S = \sum_i P_{X_i} \widetilde{\t_i} \in {\cal I}_{n-1}^1(U) \), let \( \overline{S} =  \sum_i E_{X_i} \widetilde{\t_i} \in {\cal I}_{n}^0(U)  \), the ``infinitesimal fill'' of the integral dipole \( (n-1) \)-chain \( S \).

\begin{prop}\label{prop:more}
If \( S \) is an integral dipole chain, then \( |S| = |\overline{S}| = |E_X S| = |P_X S| \). 
\end{prop}
 
\begin{proof} 
This follows since \( |E_X \widetilde{\t}| = |P_X \widetilde{\t}| = |\widetilde{\t}| = \t \) and there is no cancellation when forming finite sums since all integral monopole cells are limits of positively  oriented Dirac $3$-chains.
\end{proof}    
    
	 	 All of the surfaces in the figures provide examples of  integral monopole and dipole surfaces, as do all soap films observed in nature with the simple structure observed by Plateau, as well as all smooth images of disks,  embedded orientable surfaces, nonorientable surfaces, and surfaces with multiple junctions.   The boundary of an integral dipole surface is a well-defined dipole curve, and will also be integral in our constructions below.   
	  
	Lipid bilayers and physical soap films are naturally modeled by integral dipole surfaces because of the hydrophobic effect\footnote{  ``Natural bilayers are usually made mostly of phospholipids, which have a hydrophilic head and two hydrophobic tails. When phospholipids are exposed to water, they arrange themselves into a two-layered sheet (a bilayer) with all of their tails pointing toward the center of the sheet. The center of this bilayer contains almost no water and also excludes molecules like sugars or salts that dissolve in water but not in oil. This assembly process is similar to the coalescing of oil droplets in water and is driven by the same force, called the hydrophobic effect. Because lipid bilayers are quite fragile and are so thin that they are invisible in a traditional microscope, bilayers are very challenging to study'' 
	\url{http://en.wikipedia.org/wiki/Lipid_bilayer} }.

		\begin{itemize}
			\item The surface in Figure \ref{soap.jpg} is a union of a torus and an annulus, and can be represented by an integral dipole surface using integral dipole representatives of each, and creating a triple junction along the curve in the dipole torus where it meets the dipole annulus.
			\item  The integral dipole representative of the M\"obius strip has positive orientation as seen in  the third drawing of Figure \ref{AllThree.jpg}.    Its boundary is supported in the bounding curve of the M\"obius strip.  
			\item The ``Y-problem'' geometric measure and integration theories have faced is the following:  Consider the chain  \(   \sum_{i=1}^3 \widetilde{\t_i} \) of three  oriented affine cells meeting only along a mutual edge \( L \) at 120 degrees as in Figure \ref{fig:YProblem}.  Then \( |\p \sum_{i=1}^3 \widetilde{\t_i}| \cap L = L \). This makes the boundary operator not very useful when dealing with triple junctions of cellular chains.  However, if we replace each \( \widetilde{\t_i} \) with \( P_{X_i} \widetilde{\t_i} \), then \( |\p \sum P_{X_i}\widetilde{\t_i}| \cap L  \) is just the union of the two endpoints of \( L \) as one would hope.  

		\item  By allowing non-orthogonal vector fields \( X \) in our definition of integral dipole cells, we may construct models for multiple junctions of a surface meeting in arbitrary angles.   

		\end{itemize} 

		\begin{figure}[htbp]
		 	\centering
		 		\includegraphics[height=4in]{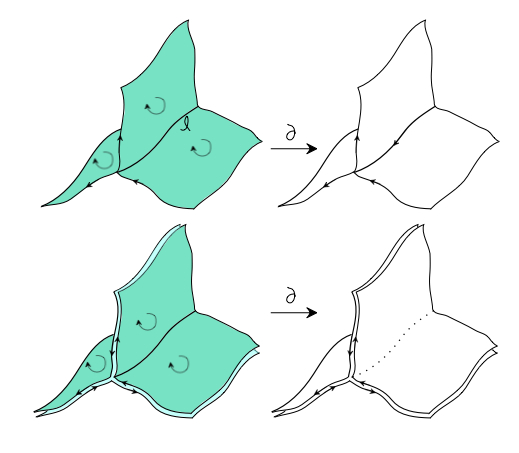}
		 	\caption{The Y-problem resolved}
		 	\label{fig:YProblem}
		 \end{figure}  

\begin{figure}[htbp]
	\centering
		\includegraphics[height=2in]{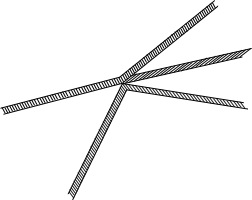}
	\caption{Multiple branches}
	\label{fig:branches}
\end{figure}
 
 \section{Geometric Poincar\'e Lemma} 
\label{sec:poincar'e_lemma}
  
\subsection{Cartesian wedge product} 
\label{sub:cartesian_product}  
 
Suppose \( U_1  \subset \R^n \) and \( U_2 \subset \R^m \) are open.  Let \( \iota_1:  U_1 \to U_1 \times U_2  \) and \( \iota_2: U_2 \to U_1 \times U_2   \) be the inclusions \(  \iota_1(p) = (p,0)  \) and \(  \iota_2(q) = (0,q) \).   Let \( \pi_1:U_1 \times U_2 \to U_1 \) and \( \pi_2:U_1 \times U_2 \to U_2 \)  be the projections \( \pi_i(p_1,p_2) = p_i \), \( i = 1, 2 \).
Let \( (p;\a) \in {\cal A}_k(U_1) \) and \( (q;\b) \in {\cal A}_\ell(U_2) \). Define \( \hat{\times}: {\cal A}_k(U_1)
\times {\cal A}_\ell(U_2) \to {\cal A}_{k+\ell}(U_1 \times U_2) \) by \[ \hat{\times}((p; \a), (q;
\b)) := ((p,q); \iota_{1*}\a \wedge \iota_{2_*}\b) \] where \( (p;\a) \) and \( (q; \b) \) are $k$- and $\ell$-elements, respectively, and extend bilinearly. We call \( P \hat{\times} Q := \hat{\times}(P,Q)\) the \emph{Cartesian wedge product} of \( P \) and \( Q \).

\begin{thm}\label{thm:GG} Cartesian wedge product \( \hat{\times}: \hB_k^r(U_1) \times
\hB_\ell^s(U_2) \to \hB_{k+\ell}^{r+s}(U_1 \times U_2) \) is associative,
bilinear and continuous for all open sets \( U_1 \subseteq \R^n, U_2 \subseteq \R^m \) and
satisfies
 \begin{enumerate} 
		\item \( \|J\hat{\times} K\|_{B^{r+s, U_1 \times U_2}} \le \|J\|_{B^{r,U_1}} \|K\|_{B^{s,U_2}} \);
		\item \( \|J \hat{\times} \widetilde{(a,b)}\|_{B^{r,U_1 \times \R}} \le  |b-a|\|J\|_{B^{r,U_1}} \) where \( \widetilde{(a,b)} \) is a \( 1 \)-chain representing the interval \( (a,b) \).
		\item \begin{align*} \p(J \hat{\times} K) = \begin{cases} ( \partial J) \hat{\times} K + (-1)^k J
\hat{\times} ( \partial K), &k > 0, \ell >0 \\ ( \partial J) \hat{\times} K, &k > 0, \ell = 0\\ J
\hat{\times} ( \partial K), &k = 0, \ell > 0 \end{cases} 
\end{align*} 
		\item \( J \hat{\times} K = 0
\) implies \( J = 0 \) or \( K = 0 \);  
		\item \( (p; \s \otimes \a) \hat{\times} (q; \t \otimes \b) = ((p,q); \s
\circ \t \otimes \iota_{1*}\a \wedge \iota_{2*}\b) \);
		\item  \( (\pi_{1*}\o \wedge \pi_{2*}\eta)(  J \hat{\times}   K) = \o(J)\eta(K) \)  for \( \o \in \B_k^r(U_1), \eta \in \B_\ell^s(U_2) \);
		\item \( |J \hat{\times} K| = |J| \times |K| \).
\end{enumerate}
\end{thm}
  See \cite{OC} (Proposition 9.1.2 and  Theorem 9.1.3).

\subsection{Geometric Poincar\'e Lemma} 
\label{sub:poincare}
  We say that an open set \( U \subset \R^n \) is \emph{contractible} if there exists a map \( F\in M^r([0,1] \times U,U) \) and a point \( p_0 \in U \) with \( F(0,p) = p_0 \) and \( F(1,p) = p \) for all \( p \in U \). According to Theorems \ref{thm:A} and \ref{thm:GG} \( F_*:\hB_{k+1}^r((0,1) \times U)  \to \hB_{k+1}^r(U)   \) is continuous with 
\[ \|F_*(\widetilde{(0,1)} \hat{\times} A) \|_{B^r}   \le n2^r\max\{1, \rho_r(F) \}  \|A\|_{B^r}.  \] 

\begin{defn} \label{cone}   
 	Define the \emph{\textbf{cone operator}} \( \kappa: \hB_k^r(U)  \to \hB_{k+1}^r(U) \) by \( \kappa(J) := F_*(\widetilde{(0,1)} \hat{\times} J) \).   
\end{defn}

\begin{thm}\label{thm:H} \( \ch \)
  \( \kappa: \hB_k^r(U) \to \hB_{k+1}^r(U) \) is a continuous linear map satisfying \( \|\kappa J\|_{B^r} \le R \|J\|_{B^r} \) for all differential chains \( J \in   \hB_k^r(U)  \), and for all \( 0 \le k \le n-1 \) and \( r \ge 1 \).  Furthermore \( \kappa \p + \p \kappa = I \).     
\end{thm} 

\begin{proof} 
We know \( \kappa \) is continuous since pushforward and Cartesian wedge product are continuous.    See Theorem \ref{thm:A} to establish the inequality.

Furthermore, by Theorem \ref{thm:GG}
\begin{align*}
 \kappa \p (p;\a) + \p \kappa (p;\a) &= F_* (\widetilde{(0,1)} \hat{\times} \p (p;\a)) + \p F_* (\widetilde{(0,1)} \hat{\times} (p;\a))\\&=  F_* (\widetilde{(0,1)} \hat{\times} \p (p;\a)) +   F_*\p (\widetilde{(0,1)} \hat{\times} (p;\a))\\&=  F_* (\widetilde{(0,1)} \hat{\times} \p (p;\a)) +   \left(F_*  (\p \widetilde{(0,1)} \hat{\times} (p;\a)) - F_* (  \widetilde{(0,1)} \hat{\times} \p (p;\a))\right)  \\&= 	 F_*  (\p \widetilde{(0,1)} \hat{\times} (p;\a))   \\&= F_*((1;1)  \hat{\times} (p;\a))  - F_*((0;1)  \hat{\times} (p;\a)).
\end{align*} 

Since \(  F(1,p) = p \) for all \( p\in U \), we deduce \( F_*((1;1)  \hat{\times} (p;\a)) = (p;\a) \), and since \( F(0,p) = p_0 \), we obtain \( F_*((0;1)  \hat{\times} (p;\a)) = 0 \).  It follows that  \( \kappa \p + \p \kappa \) is the identity on Dirac chains, and thus is the identity operator on \( \hB_k^r(U) \). 

\end{proof}

\begin{cor}\label{thm:poincarelemma}[Geometric Poincar\'e Lemma] 
Let \( U_1 \subset U_2 \subset \R^n \) be open subsets where \( U_1 \) is \( B^r \) contractible in \( U_2 \).  If \( J \in  \hB_k^r(U_1) \) with \( \p J = 0 \), then there exists \( C \in \hB_{k+1}^r(U_2) \) with \( \p C = J \) for all \( 1 \le k \le n \) and \( 1 \le r \le \i \).  
\end{cor}     

 \begin{proof} 
  Let \( C = \kappa J \).   Note that \( C \) is unique up to addition of a  chain boundary since $J = \p (\kappa J + \p C).$  
 \end{proof}

The classical Poincar\'e Lemma for differential forms of class \( {\cal B} \) follows since the homotopy operator \( H \) given by \( H \o:= \o \kappa \) is continuous.  (One can also derive the classical formula  \( (H \o) (p;\a) =  \int_0^1 i_{\p/\p t}    \o(F_t(p);  \a)   dt \) directly from our definition of \( \kappa \), and we leave this for interested readers as an exercise.)

\begin{cor}
  \label{cycle}   There does not exist a nonzero differential $n$-cycle  \( J \in \hB_n^r(U) \) supported in a contractible open set \( U \) of a smooth $n$-manifold $M$ for all \( n > 0 \).  
\end{cor} 

\begin{proof}
  If $\partial J = 0$, then   $\kappa\partial J = 0$ since \( \kappa \) is a continuous operator.   However, $\kappa J = 0$ since every $(n+1)$-chain in $\R^n$ is degenerate. Thus   $$J =
  \partial \kappa J + \kappa
  \partial J = 0.$$
\end{proof} 
 
The next result is included for completion, and is not used in this paper:
\begin{cor} [General Intermediate Value Theorem]\label{cor:ivt} Suppose $F: \R^m \to \R^n$ is a smooth map where $1 \le n \le m$,  $J \in \hB_n^r(\R^m)$  and $K\in \hB_n^r(\R^n)$    with  $| F_*J| \cup |K|$ a compact subset of   a contractible open set $U$ of $\R^n$. 
	Then $$F_* (\partial J) = \partial K \iff F_*J = K.$$
\end{cor} 

	\begin{proof} This follows from Lemma \ref{cycle} and since the boundary and pushforward operators are continuous and commute.  
	\end{proof}

				\section{The part of a chain in a compatible cube}  
				\label{sub:the_part_of_a_chain_in_a_compatible_cube}

				Suppose \( J \in \hB_n^1(U)  \) and \( Q  \subset  U \) is an \( n \)-cube.  Differential chains have well behaved inclusions, but restrictions can be problematic.  A simple example is the dipole \( D= (1; e_1 \otimes 1) \).  It is not possible to define the part of \( D \) in \( (0,1) \subset \R^1 \) in some continuous way because   Dirac chains of order zero and limiting to \( D \) can be chosen supported entirely inside \( (0,1) \), outside \( (0,1) \), or partly inside and partly outside \( (0,1) \), and these parts can limit to many different chains, if at all.  However, we can define the ``part of \( J \)'' in an \( n \)-cube \( Q \) whose faces are not contained in a set \( Z_J \) with Lebesgue measure zero and consisting of a union of affine hyperplanes to be determined below.  

\begin{defn} 
	 	If \( A = \sum (p_i;   \a_i) \in \A_n^0(U) \) is a Dirac $n$-chain  in \( U \) and \( E \) is a subset of \( U \), define the \emph{\textbf{part of \( A \) in \( E  \)} } by \( A\lfloor_E := \sum (p_{i_j};   \a_{i_j}) \) where \( p_{i_j} \in E \). 
	
\end{defn}
   If \( A, A' \) are Dirac \( n \)-chains and \( c \in \R \), then 
\begin{equation} \label{star}
	 (cA)\lfloor_E = c (A\lfloor_E) \mbox{ and } (A+A')\lfloor_E =  A\lfloor_E + A'\lfloor_E .
\end{equation}   

\begin{defn}
	For   \( 0 \le \ell \le n \)  and  \( t \in \R \), let  \( H_t^\ell(U) = \{x \in U: x_{\ell} < t\} \)  be the open \emph{\textbf{half-space}} where \( x = \sum_{i=1}^n x_i e_i \).   
\end{defn}    
   In the next two lemmas, let \( Q = I_1 \times \cdots \times I_n \) be an open coordinate \( n \)-cube with side length \( T \).  
													\begin{lem}\label{lem:first}
											   Suppose $D \in  {\cal A}_n^0(Q)$ is a Dirac $n$-chain and \( 1 \le \ell \le n \). Then  $$ \int_a^b \|D \lfloor_{H_t^\ell}\|_{B^1}dt \le   (2T+1)  \|D\|_{B^1}.$$ 
													\end{lem}  
												 
											\begin{proof}    Assume without loss of generality that \( \ell = 1 \).    Write \( I_1 = (a,b) \).  By assumption \( b-a = T \).  
											Suppose \( \D_u(q;\b) = (q + u; \b) - (q; \b) \) where   \( q = \sum_{i=1}^n q_i e_i \in Q \),    \( u \in \R^n \) is a multiple of \( e_1 \) with \( q' = q+u \in Q \), and \(\b \in \L_n(\R^n) \).   In particular,  \( 0 < \|u\| < T \). We first prove 
											\begin{equation}\label{intge} \int_a^b \| \D_u (q; \b) \lfloor_{H_t^\ell}\|_{B^1}dt \le  (T+1) \|u\| \|\b\|. 
											\end{equation}   
											     The integral splits into three parts: 
											\begin{align}\label{intge2}
												 \int_a^b \|\D_u(q;\b) \lfloor_{H_t^\ell}\|_{B^1}dt  &=    \int_a^{q_1} \| \D_u(q;\b) \lfloor_{H_t^\ell}\|_{B^1}dt  \\&+  
												\int_{q_1}^{q_1'} \| \D_u(q;\b) \lfloor_{H_t^\ell}\|_{B^1}dt  \notag\\&+ 
												 \int_{q_1'}^b \| \D_u(q;\b) \lfloor_{H_t^\ell}\|_{B^1}dt. \notag 
											\end{align}         
										   The first integral is zero, the second is bounded by \( \|u\| \|\b\|  \), and the third is bounded by \( T \|\D_u ( \b)\|_{B^1} \le T  \|u\| \|\b\| \).  This gives us an upper bound of \( (T+1) \|u\|\|\b\| \) for the LHS of   \eqref{intge2}.     
										
										Now suppose \( w \in \R^n \).  Then \( w = u + v \) where \( u \) is in the direction of \( e_1 \) and \( v \) is orthogonal to \( e_1 \).  Since  \(  \int_a^b \| \D_v (q;  \b) \lfloor_{H_t^\ell}\|_{B^1}dt \le  T \|v\| \|\b\|  \), it follows from the triangle inequality that \(  \int_a^b \| \D_w (q;  \b) \lfloor_{H_t^\ell}\|_{B^1}dt \le   \int_a^b \| \D_u (q;  \b) \lfloor_{H_t^\ell}\|_{B^1}dt  +  \int_a^b \| \D_v (q;  \b) \lfloor_{H_t^\ell}\|_{B^1}dt \le  (2T+1) \|w\| \|\b\|   \).

												Likewise, 
												\begin{align}\label{bdd}
												  \int_a^b \| (p; \a) \lfloor_{H_t^\ell}\|_{B^1}dt \le  T  \|\a\| 
												\end{align} 
												  for all \( p\in Q \mbox{ and }  \a \in  \L_n (\R^n) \).    

											Now let \( D \in  {\cal A}_n^0(Q)  \) and \( \e > 0 \).  We write \( D =   \sum_{s=0}^j (p_s;  \a_s)  + \sum_{i=1}^{m}   \D_{w_{i}}  (q_{i};  \b_{i}) \) where each  \(  |\D_{w_{i}}   (q_{i};   \b_{i})| \) and \( |(p_s;  \a_s)| \subset Q \) and  \(  \|D\|_{B^1} \le  \sum_{s=0}^j   \|\a_s\| +    \sum_{i=1}^{m }   \|w_{i}\|  \|\b_{i}\|  \le \|D\|_{B^1} + \e \).    

 											  The bounds \eqref{intge2} and \eqref{bdd} yield 
											\begin{align*}
											\int_a^b \|D\lfloor_{H_t^\ell}\|_{B^1}dt &\le  \sum_{s=0}^j\int_a^b \|(p_s;  \a_s)\lfloor_{H_t^\ell} \|_{B^1} dt +    \sum_{i=1}^{m} \int_a^b  \|\D_{w_{i}}    (q_{i};   \b_{i})\lfloor_{H_t^\ell}\|_{B^1} dt \\&\le (2T+1)  \left( \sum_{s=0}^j  \|\a_s\| +   \sum_{i=1}^{m} \|w_{i}\| \|\b_{i}\|  \right) \\&<  (2T+1) (\|D\|_{B^1} + \e).   
											\end{align*}   The Lemma follows.

										  	\end{proof}   

											\begin{lem}\label{lem:beyond1}     
												If \( \{D_i\} \subset \A_n^0(Q) \) is a sequence of Dirac \( n \)-chains with  \( \sum_{i=1}^\i \|D_i\|_{B^1} < \i \), then   \( \sum_{i=1}^\i \|D_i\lfloor_{H_t^\ell} \|_{B^1} < \i \) for almost all \( t \in \R \) and each \( 0 \le \ell \le n \).
											\end{lem}     

											\begin{proof}
												Apply Lemma \ref{lem:first}  to deduce 
											\begin{align*}
											  \int_a^b  \left( \sum_{i=1}^N \|D_i\lfloor_{H_{t}^\ell}\|_{B^1} \right) dt = \sum_{i=1}^N \int_a^b  \left( \|D_i\lfloor_{H_{t}^\ell}\|_{B^1}dt \right)    \le   (2T+1)\sum_{i=1}^N   \|D_i\|_{B^1} \le   (2T+1)\sum_{i=1}^\i   \|D_i\|_{B^1} < \i.
											\end{align*} 
											  Let \( N \to \i \).     It follows  from the monotone convergence theorem that  \( \int_a^b  \left( \sum_{i=1}^\i \|D_i\lfloor_{H_{t}^\ell}\|_{B^1} \right) dt    < \i\) and  	thus \( \sum_{i=1}^\i \|D_i\lfloor_{H_t^\ell}\|_{B^1} \) converges a.e. \(a \le t \le b \).
											\end{proof}

								  Lemmas \ref{lem:first} and \ref{lem:beyond1}   hold for both open and closed \emph{half-spaces}  in  \( \R^n \). 
										
							\begin{defn}
								 An $n$-cube \( Q \) is \emph{partly open and closed} if it contains some of its \( (n-1) \)-faces, but possibly not all. If \( Z \) is a union of affine hyperplanes of \( \R^n \), we say that an $n$-cube  \( Q  \) is \( Z \)-\emph{compatible} if no face of \( Q \) is contained in \( Z \).      
							\end{defn}			   
								
									    Let  \( \O_R \) be a coordinate \( n \)-cube containing the origin of side length \( R \). 

										\begin{thm}\label{thm:JINU1}
										   For each  \( J \in \hB_n^1(\O_R)  \), there exists a union of affine hyperplanes   \( Z_J \subset \R^n \)  of Lebesgue measure zero set  such that if \( Q \subset \O_R \) is a \( Z_J \)-compatible partly open and closed coordinate $n$-cube, then there exists a unique differential \( n \)-chain \( J\lfloor_Q  \in \hB_n^1(\O_R) \) such that \( |J\lfloor_Q| \subset \overline{Q} \) and \( \cint_{J\lfloor_Q} \o = \cint_J \o \) for all \( \o \in \hB_n^1(\O_R) \) with support in \( Q \).     
										If \( J_i \to J \) in \( \hB_n^1(\O_R) \), then  \( J_i\lfloor_Q  \to J\lfloor_Q \) in \(  \hB_n^1(\O_R) \) for all \( Q \subset \O_R \) which are \( Z_J \cup_i Z_{J_i} \)-compatible.  Furthermore, \( (cJ)\lfloor_Q = c(J\lfloor_Q) \) for all \( c \in \R \) if \( Q \) is \( Z_J \)-compatible, and  \( (J+K)\lfloor_Q = J\lfloor_Q + K\lfloor_Q \) if \( Q \) is \( Z_J \cup Z_K \)-compatible.
										\end{thm}	                   

										\begin{proof} We first prove the result for half-spaces in the direction of \( e_\ell \) for each \( 1 \le \ell \le n \).   There exist \( D_j \to J \) in \( \hB_n^1(\O_R) \) with \( D_j \in {\cal A}_n^0(\O_R)  \) and \( \|D_j - D_{{j+1}}\|_{B^1} < 2^{-j} \).  Therefore, \[ \sum_{j=1}^\i  \|D_j - D_{{j+1}}\|_{B^1} < \i. \]    By Lemma \ref{lem:beyond1}, 
											\( \sum_{j=1}^\i  \|(D_j - D_{{j+1}})\lfloor_{H_t^\ell}\|_{B^1}  < \i \) for a.e.  \( t \).  Let \( Z_J(\ell) \) be the union of \( \fr(H_t^\ell) \) for all \( t \) for which \( \sum_{j=1}^\i  \|(D_j - D_{{j+1}})\lfloor_{H_t^\ell}\|_{B^1}   \) diverges.    If \( \fr(H_t^\ell)  \notin Z_J(\ell) \), then \(\{D_j\lfloor_{H_t^\ell}\}_j \) forms a Cauchy sequence in \( \hB_n^1(\O_R) \) and we may define 
											\[
												J\lfloor_{H_t^\ell} := \lim_{j \to \i} D_j\lfloor_{H_t^\ell} \in   \hB_n^1(\O_R).
											\]  If \( \{ D_i' \} \) is another sequence which has the above properties and tends to \( J \), then by taking subsequences, we may assume that \( \sum \| D_i - D_i'\|_{B^1} < \i \).  If the frontier of \( H_t^\ell \) is not contained in \( Z_J(\ell) \cup_i Z_{D_i}(\ell) \),  then \( D_i\lfloor_{H_t^\ell} \) and \( D_i'\lfloor_{H_t^\ell} \) tend to the same limit as \( i \to \i \).  The half-planes can be chosen to be either open or closed, because we are ignoring everything in their frontiers.

											 Let \( Z_J := \cup_{\ell=1}^n Z_J(\ell) \).  	
											
											Now suppose		\( Q \subset \O_R \) is a partly open and closed coordinate $n$-cube and \( Z_J \)-compatible.  Then \( Q = \cap_\ell H_{t_\ell}^\ell \) and   \(   \fr(H_{t_\ell}^\ell) \cap Z_J(\ell) = \emptyset  \) for \( 1 \le \ell \le n \).    Let \( J\lfloor_{Q} := ( \cdots (J\lfloor_{H_{t_1}^1})\lfloor_{H_{t_2}^2} \cdots)\lfloor_{H_{t_n}^n} \).
											
										 	 The integral condition holds, for if \( \o  \in  {\cal B}_n^1(\O_R) \) is supported in \( Q \), then  
											\( \cint_{D_i\lfloor_Q} \o = \cint_{D_i} \o \).   Therefore, \( \cint_{J \lfloor_Q} \o = \lim_{i \to \i} \cint_{D_i\lfloor_Q} \o = \lim_{i \to \i}  \cint_{D_i} \o  = \cint_J \o \).    
											
											We next show \( |J\lfloor_Q| \subset \overline{Q} \): 	 Suppose there exists \( p \in |J\lfloor_{Q}| \) and \( p \notin  \overline{Q} \).  Let \( Q' \) be a compatible \( n \)-cube containing \( p \) and   \( \overline{Q'} \cap  \overline{Q} = \emptyset \).  Since \( p \in |J\lfloor_{Q}| \) there exists a differential form \( \eta \) supported in \( Q' \) and with \(   \cint_{J\lfloor_{Q}} \eta \ne 0 \).    However,  \( \overline{Q'} \cap  \overline{Q} = \emptyset \) implies \( \cint_{J\lfloor_{Q}} \eta = \lim_{i \to \i} \cint_{D_i\lfloor_{Q}} \eta  = 0 \).

 									 	Linearity  \eqref{star} carries over to differential chains in almost every half-space since the integral is bilinear and continuous, and finally to differential chains in compatible \( n \)-cubes.   We deduce  \( (cJ)\lfloor_Q = c(J\lfloor_Q) \) for all \( c \in \R \) and  \( (J+K)\lfloor_Q = J\lfloor_Q + K\lfloor_Q \).
											\end{proof}

   Let \( {\cal H}^d \) denote \( d \)-dimensional Hausdorff measure.  

	\begin{prop}\label{cor:hausdQ}   If \( S \in {\cal I}_{n-1}^1(\O_R) \) is an integral dipole chain, then \( \cint_{\overline{S}} dV = {\cal H}^{n-1}(|S|) \) and \[ \cint_{\overline{S}\lfloor_Q} dV = {\cal H}^{n-1}(|S| \cap Q) \] for all \( Z_S \)-compatible \( n \)-cubes \( Q \).
	\end{prop}

	\begin{proof}  We can write \( S = \sum_{i=1}^s P_{X_i} \widetilde{\t_i} \in {\cal I}_{n-1}^1(\O_R) \) where the \( (n-1) \)-cells \( \{\t_i\} \) are non-overlapping and closed.  Thus \( |S| = \cup_{i=1}^s \t_i \).   Furthermore, \( \overline{S}\lfloor_Q = \sum_{i=1}^s E_{X_i} \widetilde{\t_i} \lfloor_Q \).  Using the integral relation  \eqref{EXthm}, Theorem \ref{thm:B},  the assumption that  the component of \( X_i \) orthogonal to \( \t_i \) is unit,  the definition of Hausdorff measure for smoothly embedded cells,  additivity of Hausdorff measure for non-overlapping sets with smooth boundaries, and the last part of Theorem \ref{thm:JINU1}, we have
 \begin{align*}
  \cint_{\overline{S}\lfloor_Q} dV &= \sum_{i=1}^s \cint_{E_{X_i} \widetilde{\t_i} \lfloor_Q} dV = \sum_{i=1}^s \cint_{  \widetilde{\t_i} \lfloor_Q} i_{X_i} dV = \sum_{i=1}^s\int_{{\t_i \cap Q}} i_{X_i} dV  \\&=  \sum_{i=1}^s{\cal H}^{n-1}(\t_i \cap Q)  = {\cal H}^{n-1}((\cup_i \t_i) \cap Q)  = {\cal H}^{n-1}(|S| \cap  Q). 
 \end{align*} 
  The last integral is the Riemann integral.   	

The proof that \( \cint_{\overline{S}} dV = {\cal H}^{n-1}(|S|) \) is similar.
	\end{proof}

			\begin{defn} 
			We say \( J \in \hB_n^r(\O_R) \) is \emph{\textbf{positively oriented}} if \( \cint_{J \lfloor_Q} f dV \ge 0 \) for all \( Z_J \)-compatible \( n \)-cubes \( Q \subset \O_R \) with \( f \in \B_0^r(\O_R) \) and \( f \ge 0 \). 
			\end{defn} 
			
			For example, all integral monopole \( n \)-chains are positively oriented since they are limits of positively oriented Dirac \( n \)-chains. 

				\begin{lem}\label{fdV}
				If \( J \in \hB_n^r(\O_R) \) is positively oriented, then \( \|J\|_{B^r} = \cint_J dV \) for all \( r \ge 1 \).
				\end{lem}

				\begin{proof} 
 		 	We may assume that \( f \ge 0 \) since \( J \) is positively oriented and \(  \|J\|_{B^r} = \sup_{f \ne 0} \frac{\cint_J fdV}{\|f\|_{B^r}} \).  By definition of the \( B^r \) norm on functions, \( |f(x)| \le  \|f\|_{B^r} \) which implies  \( \frac{|\cint_J f dV|}{\|f\|_{B^r}} =  \cint_J \frac{f}{\|f\|_{B^r}} dV \le  \cint_J dV \).  Thus \( \|J\|_{B^r} = \sup \frac{|\cint_J f dV|}{\|f\|_{B^r}}  \le \cint_J dV   \).  On the other hand,  \( \cint_J dV \le \|J\|_{B^r} \).  
				\end{proof}

					 \section{The volume functional used to compute area} 
					 \label{ssub:area}

				 	\begin{defn}
						The \emph{\textbf{area}} of an integral dipole $(n-1)$-cell  \( P_X  \widetilde{\t} \)   is defined by  \[  A(P_X \widetilde{\t}) := \cint_{ E_X \widetilde{\t}} dV   \] where \( dV \) is the volume form in \( \R^n \).   Since  \( E_X  \widetilde{\t} \) is a limit of positively oriented Dirac \( 3 \)-chains, then \( A(P_X \widetilde{\t}) > 0 \) if \( \t \) is non-empty.   	If \( S = \sum_{i=1}^s   P_{X_i} \widetilde{\t_i} \) is an integral dipole \( (n-1) \)-chain, define  	
						\begin{equation}\label{A(S)}
								A(S) := \cint_{\overline{S}} dV.
							\end{equation} 
					 	\end{defn}

				In order to extend definition (\ref{A(S)}), we need a way to define \( \overline{S} \)  using our continuous operators so that we can take limits.  Proposition \ref{lem:filledmore} does just this:  

				\begin{prop}\label{lem:filledmore}
				If \( S = \sum_{i=1}^s P_{X_i} \widetilde{\t_i} \) is an integral dipole chain with \( \p S = P_Y \widetilde{\g} \), then \[ \overline{S} = \kappa(S) \] where \( \overline{S} = \sum_{i=1}^s E_{X_i} \widetilde{\t_i}  \). Moreover, \( \|\overline{S}\|_{B^1} \le \|S\|_{B^2} \).  
				\end{prop}

				\begin{proof} By Theorem \ref{thm:H}  \( \kappa \p + \p \kappa = I \).  Since \( \p S = P_Y \widetilde{\g} \), it follows that \( \p \overline{S} = S - E_Y \widetilde{\g} = \p \k (S - E_Y \widetilde{\g}) \).   Using Corollary \ref{cycle}, we obtain \( \overline{S} = \k (S - E_Y\widetilde{\g}) = \k(S) \).

        Since \( \overline{S} \) is positively oriented we may apply Lemma \ref{fdV} to deduce 
				\begin{align*}
				 \|\overline{S}\|_{B^1} =  \cint_{\overline{S}} dV = \cint_{\k(S)} dV = \cint_S H dV \le \|S\|_{B^2} \|HdV\|_{B^2} 
				\le \|S\|_{B^2}.
				\end{align*} 

				\end{proof}

					\begin{thm}\label{prop:areas} If \( S \) is an integral dipole \( 2 \)-chain with \( \p S = P_Y \widetilde{\g} \), then
					\[ A(S)   = \cint_{\k(S)} dV. \]
					\end{thm}

					\begin{proof} 
					This follows from  Proposition \ref{lem:filledmore}.      
					\end{proof}

				For \( S \in \overline{{\cal I}_2^1(\O_R)} \), define \( A(S):=  \cint_{\k(S)} dV. \)  It follows from the continuity of \( \k \) (Theorem \ref{thm:H}) and joint continuity of the integral that \( A:\overline{{\cal I}_2^1(\O_R)} \to \R \) is continuous.

\section{Existence of area minimizers for surfaces spanning a smoothly embedded closed curve in $ \R^3 $} 
\label{sub:surfaces} 
\subsection{A complete space of spanning dipole surfaces} 
\label{ssub:surfaces}

  For the rest of the paper, we will assume \( n = 3 \) for visual simplicity.  However, all results extend immediately to codimension one surfaces with multiple junctions in \( \R^n \) for \( n \ge 2 \).   
We indicate how to extend to arbitrary codimension by the first remark in \eqref{remarksend}.

Recall the space of integral dipole \( 2 \)-chains \[ {\cal I}_2^1(\O_R)  \subset \overline{{\cal I}_2^1(\O_R)} \subset \ch_2^1(\O_R). \]

   Let  \( \g:S^1 \to \O_R \)  be  a smoothly embedded closed curve, and denote \( \g = \g(S^1) \).   By Sard's theorem a.e. \( q \in \O_R\backslash \g \) the vectors \(  q-p \) are transverse to \( \g \) for all \( p \in \g \).  Choose such a point \( q \).   (Theorem \ref{thm:indep} shows that our solution to Plateau's problem is independent of the choice of \( q \).)     
Let   \( \widetilde{\g} \in \ch_1^0(\O_R)  \) be the $1$-chain representing \( \g \) (see Theorem \ref{thm:B}).  Let \( Y(p) = Y_q(p) = (q-p)/\|q-p\| \) for all \( p \in \g \).  Since \( \g \) is smooth, we may use the Whitney Extension Theorem to extend \( Y \) to a smooth unit vector field on a neighborhood of \( \g \).  We are only interested in the restriction of \( Y \) to \( \g \) (where it is assumed to be transverse), and the fact that \( Y \) extends to a smooth vector field  in a neighborhood of \( \g \).  It follows that if \( Y_1 \) and \( Y_2 \) are any two such vector fields, then \( P_{Y_1}\widetilde{\g} = P_{Y_2}\widetilde{\g} \).

\begin{defn}\label{def:span}
 	We say that   \( S \in \overline{{\cal I}_2^1(\O_R)}  \)  \emph{\textbf{spans}} \( \g \) (with respect to \( q \)) if 

\begin{itemize}
	\item \( \p S = P_Y \widetilde{\g} \);
  \item If \( \g' \) is  a smoothly embedded closed curve linking \( \g \) with linking number one, then \( \g' \cap |S| \ne \emptyset \).
\end{itemize}
\end{defn}

The first condition assures us that the support of the boundary of \( S \) is \( \g \), the second  that  there are no holes in \( S \) as in Figure \ref{soap.jpg}.

\begin{prop}\label{thm:spanwell}
Span is a closed condition on \( \overline{{\cal I}_2^1(\O_R)} \).
\end{prop}

				\begin{proof} 
							Suppose \( S_i \to S \) where \( S_i \in {\cal I}_2^1(\O_R) \) and \( |S_i| \) spans \( \g \).   Suppose \( |S| \) does not span \( \g \).   	Then there exists a simple link 	\( \g'\subset \O_{R} \)  of \( \g \) and \( |S|\cap \g' = \emptyset \).   Let \( N_\d \) be a tubular \( \d \)-neighborhood of \( \g' \)  with \( N_\d \cap |S|= \emptyset \) and which is a union of non-overlapping \( Z_S \)-compatible cubes \( N_\d = \cup_j Q_j \).  (Recall that the cubes \( Q_j \) can be partly open or closed.)  Then  \( N_\d \) contains a smoothly embedded solid torus \( N \) which is a union of links \( \ell \) parallel to \( \g' \).   	 Let \( D \) be a transverse section of \( N \) so that each \( \ell \) meets \( D \) once.  	Since \( |S_i| \) spans \( \g \), Hausdorff measure can only decrease on projection, and Proposition \ref{cor:hausdQ} we deduce 
							\begin{align*}
							 \cint_{\overline{S_i}\lfloor_{N_\d}} dV = \sum_j \cint_{\overline{S_i}\lfloor_{Q_j}} dV = \sum_j {\cal H}^2(|S_i| \cap Q_j) = {\cal H}^2(|S_i| \cap N_\d)  \ge {\cal H}^2(|S_i| \cap N) \ge  {\cal H}^2(D)  > 0.
							\end{align*} 
							However, \( N_\d \cap |S| = \emptyset \) implies \( \overline{S}\lfloor_{N_\d} = 0 \), contradicting Theorem \ref{thm:JINU1}.   It follows that \( S \) spans \( \g \).

				\end{proof}

We now have established  definitions of \emph{surface}, \emph{area} and \emph{span} for which we can solve problems of the calculus of variations such as Plateau's problem.     
    
Let \( M_0 \) be a solution to Plateau's problem in the category of integral currents \cite{federerfleming}.  According to \cite{flemingregular} \( M_0 \) is embedded and orientable. (We could also let \( M_0 \) be the integral dipole representative of the cone from \( q \) over \( \g \), even though it may not be embedded.)   Let \( c  = {\cal H}^2(M_0) \).   

\begin{defn}\label{CALC}  Let \( {\cal S}_2(\O_R, \g, q) := \{ S \in \overline{{\cal I}_2^1(\O_R)}:  A(S) \le c \mbox{ and } S \mbox{ spans } \g \}. \) 
 \end{defn}   

    Our \emph{candidate surfaces} for Plateau's problem\footnote{Added in proof:  L.C. Evans observed that our approach to solving Plateau's problem is a linear programming problem in \(  \overline{{\cal I}_2^1(\O_R)} \):  Minimize  \(  A(S) \le c \), subject to the constraints \( \p S = P_Y \widetilde{\g} \), \( S \) is positively oriented, and there are no simple links.}  in $3$-space are the supports of elements \( S \in {\cal S}_2(\O_R, \g, q) \). We know that \( {\cal S}_2(\O_R, \g, q) \) is not empty since it contains \( P_X \widetilde{M_0} \in {\cal I}_2^1(\O_R) \) where \( X \) is a smooth vector field whose component orthogonal to \( M_0 \) is unit and  which extends \( Y \).

\begin{thm}\label{thm:complete}   Suppose \( \g \) is a smoothly embedded closed curve embedded in \( \O_R -\{q\}  \).  Then 	\( {\cal S}_2(\O_R, \g, q) \) is a complete subset of \( \overline{{\cal I}_2^1(\O_R)}   \).
\end{thm}
 
\begin{proof}  This follows since area is continuous and span is closed in \( \overline{{\cal I}_2^1(\O_R)}   \) by Proposition \ref{thm:spanwell}. 

\end{proof}

  \begin{figure}[htbp]
  	\centering
  		\includegraphics[height=2in]{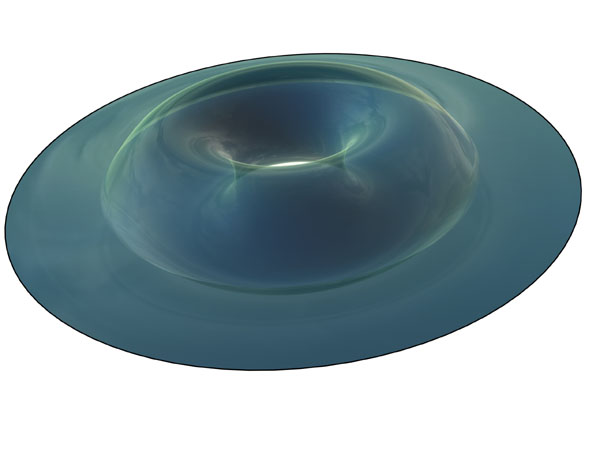}
  		\caption{A non-spanning surface of the unit circle (drawing by H. Pugh)}
  	\label{soap.jpg}
  \end{figure} 

 \subsection{A lower bound on area} 

\begin{prop}\label{pro:holes}  
Suppose \( \g \) is a smoothly embedded closed curve in \( \O_R -\{q\}  \). There exists a constant \( a_0 > 0 \) such that 	if   \( S \in {\cal S}_2(\g,\O_R, q)  \), then   \( A(S) \ge a_0 \).   
\end{prop} 

\begin{proof}  Let \( S = \lim_{i \to \i} S_i \) where \( S_i \in {\cal I}_2(\g, \O_R, q) \) spans \( \g \) and \( A(S_i) \le c \).   
  Using Sard's theorem, project \( \g \) onto an affine plane \( K \) so that the projected image  \( \pi(\g) \) has transverse self-intersections.  Let \( U \) be a component of   \( K \backslash \pi(\g) \) which meets the unbounded component of \( \pi(\g) \)  along an arc of \( \pi(\g) \).   Then the area \( a_0 \) of \( U \) is nonzero.  

It suffices to show that \( U \subset \pi(|S_i|) \) for each \( i \ge 1 \).  For then \( A(S_i) = {\cal H}^2(|S_i|) \ge a_0 \) for each \( i \), and hence \( A(S) \ge a_0 \).

   Choose \( x \in U \).   We wish to show \( \pi^{-1}(x) \cap |S_i| \ne \emptyset \).  By the definition of \( U \),  we obtain a simple link of \( \g \) by connecting the endpoints of \( \pi^{-1}(x) \cap \O_R \) with a geodesic arc in the boundary of \( \O_R \).  This link will meet \( |S_i| \) since \( S_i \) spans \( \g \). Since the link cannot  meet \( |S_i| \) in the boundary of \( \O_R \), then it must meet it in \( \pi^{-1}(x) \).  It follows that  \( x \in \pi(|S_i|) \).     

  \end{proof}  

\subsection{Compactness} 
 \begin{prop}\label{pro:tb}
Suppose \( \g \) is a smoothly embedded closed curve   embedded in \( \O_R -\{q\}  \).  Then	\( {\cal S}_2(\O_R, \g, q) \) is totally bounded in \(  \ch_2^1(\O_R) \).
\end{prop}  
 \begin{proof}
If \( P_X \widetilde{\t} \in \ch_2^1(\O_R)  \), then it is  approximated by finite sums \( \sum_{i=1}^s P_{X(q_i)}(q_i; \b_i) \) in the \( B^2 \) norm.  
 
For \(  k \in \Z, \, k \ge 1 \), let \( {\cal Q}(k) \) be all rationals \( j/2^k \) with \( j \in \Z \) and   \( 0 \le |j| \le 2^k \).    Let   \( {\cal Z}(k)   \) be the subset of   integral dipole  $2$-chains  \( \sum_{i=1}^s P_{v_i}(q_i;  \b_i) \in \ch_2^1(\O_R) \) and such that  

\begin{itemize}
	\item \( q_i   \) is a vertex of the binary lattice   with edge length \( 2^{-k} \) and subdividing \( \O_R \); 
	\item \( v_i/k \in \R^3 \)   has coordinates in \( {\cal Q}(k) \) so that \( \|v_i\| \le k \);
	\item \( \b_i \) is a $2$-vector with coordinates in \( {\cal Q}(k) \) written in terms of an orthonormal basis of \( \L_2(\R^3) \);
	\item \( \sum_{i=1}^s   M(\b_i) \le   c \).  
 \end{itemize}
It follows that  \( {\cal Z}(k) \) contains only finitely many chains.  

Let \( S  \in {\cal S}_2(\O_R, \g, q) \).   Then \( S \mbox{ spans } \g, A(S) \le c, \mbox{ and } |S| \subset   \O_R  \).  It follows that  \( S \) can be approximated by an affine dipole $2$-chain \( S' \), which can itself be approximated by a Dirac dipole $2$-chain  \( A = \sum_{i=1}^{s} P_{u_i}(p_i;   \a_i)  \in {\cal A}_2^1(\O_R) \) 
with   \( \|u_i\| \le k \),   \( \sum_{i=1}^s \|\a_i\|   \le c \).  All approximations  are in \( \ch_2^1(\O_R)  \).   

  Now \( A \) can be approximated by an element of \( {\cal Z}(k) \)  as follows:   For each dipole \( 2 \)-vector  \( P_{u_i}(p_i;    \a_i) \) we know   \( p_i \) lies in some cube \( Q' \) of the binary lattice subdividing \( \O_R \). Let \( p_i' \) be a vertex of \( Q' \).  Let \( u_i' \)   have coordinates in \( {\cal Q}(k) \) s.t.  \( \|u_i - u_i'\| \le 2^{-k}  \).  Let \( \a_i' \) be a $2$-vector with coordinates in \( {\cal Q}(k) \) s.t. \( \|\a_i- \a_i'\| < 2^{-k} \|\a_i\|) \).
   Then \( A' = \sum_{i=1}^{s} P_{u_i'}(p_i';   \a_i') \in {\cal Z}(k)  \).  It follows that 
\begin{align*}
  \| A - A'\|_{B^1} &\le  \|\Sigma_i P_{u_i}(p_i;   \a_i) - P_{u_i'}(p_i;   \a_i)\|_{B^1} +  \|\Sigma_i P_{u_i'}(p_i;   \a_i) - P_{u_i'}(p_i';   \a_i)\|_{B^1} \\& \hspace{1in}+ \|\Sigma_i P_{u_i'}(p_i';   \a_i) - P_{u_i'}(p_i';   \a_i')\|_{B^1} \\&\le   
 \sum_{i=1}^s \| P_{u_i - u_i'}(p_i;   \a_i) \|_{B^1}  +  
 \sum_{i=1}^s \| P_{u_i'}((p_i;   \a_i) - (p_i';   \a_i))\|_{B^1} + 
 \sum_{i=1}^s \| P_{u_i'}(p_i';   \a_i - \a_i') \|_{B^1}
\\&\le 
\sum_{i=1}^s \|u_i -u_i'\| \|\a_i\|  + \|u_i'\|\|p_i -p_i'\| \|\a_i\| + \|u_i'\| \|\a_i -\a_i'\|
\\&\le
k\sum_{i=1}^s 2^{-k}\|\a_i\| + (1 + 2^{-k}) 2^{-k}\|\a_i\|  + (1 + 2^{-k})  2^{-k} \|\a_i\| 
\\&\le k2^{-k} \sum_{i=1}^s \|\a_i\| + 2(1 + 2^{-k})  \|\a_i\|   
\\&\le 4k 2^{-k} \sum_{i=1}^s \|\a_i\| \\&\le 4kc 2^{-k}. 
\end{align*}
 
  This proves that \(  {\cal S}_2(\O_R, \g, q)  \) is totally bounded.
\end{proof}    
 
\begin{thm}\label{pro:com}
 Suppose \( \g \) is a smoothly embedded closed curve   embedded in \( \O_R - \{q\}  \).  Then   \( {\cal S}_2(\O_R, \g, q)  \) is compact. 
\end{thm} 

\begin{proof} 
This follows from Propositions \ref{thm:complete} and \ref{pro:tb}.                                     
\end{proof}

 \begin{thm}\label{pro:m} 
Suppose \( \g \) is a smoothly embedded closed curve   embedded in \( \O_R -\{q\}  \).   	 There exists \( S_0 \in {\cal S}_2(\O_R, \g, q) \) spanning \( \g \) with minimal area \( A(S_0) \).
\end{thm}   

\begin{proof}
 Let \( m = \inf\{A(S): S\in {\cal S}_2(\O_R, \g, q)  \} \). By Proposition \ref{pro:holes} we know \( m > 0 \). There exists a sequence \( \{S_i\} \in {\cal S}_2(\O_R, \g, q) \) such that  \( A(S_i) \to m \) as \( i \to \i \).  By compactness of \( {\cal S}_2(\O_R, \g, q) \) (Theorem \ref{pro:com}) and continuity of \( A \), there exists a subsequential limit \( S_0  \in {\cal S}_2(\O_R, \g, q)\) with \( A(S_0) = m \). 
\end{proof}

It is clearly impossible to find spanning chains with smaller area using chains not supported in \( \O_R \), but we have to prove this.  The technical difficulty is that pushforward (via a diffeomorphism) of a dipole chain can change not only area but also ``dipole distance'' between layers since \( F_*(p; u \otimes \a) = (F(p); F_*u \otimes F_* \a) \).    We can ``renormalize'' a dipole surface after we have applied pushforward, so that it becomes an integral dipole surface spanning \( \g \).

\begin{thm}\label{thm:indep}
The solution \( S_0 \) is independent of the choice of \( q \) and \( R \).
\end{thm}

\begin{proof}
	Suppose \( q \) and \( q' \) are two cone points in \( \O_R \).   Suppose there exists \( S' \in  {\cal S}_2(\O_R, \g, q') \)  with \( A(S') < A(S_0) \).   By continuity of the area functional, there exists \( C =  \sum_{i=1}^s P_{X_i} \widetilde{\t_i}  \in {\cal I}_2(\O_R, \g, q') \) that spans \( \g \), \( A(C) < A(S_0) \) and \( \p C = P_{Y'} \widetilde{\g} \).  Claim:  \( C' = C +  E_{Y-Y'} \widetilde{\g} \in {\cal I}_2(\O_R, \g, q) \)  and \( A(C +  E_{Y-Y'} \widetilde{\g}) < A(S_0) \).  We first show that   \( C' \) spans \( \g \):  But \( \p C' = \p C + P_{Y-Y'} \widetilde{\g} =  P_Y \widetilde{\g} \). Furthermore, \( |C'| = |C| \) since \( |E_{Y-Y'} \widetilde{\g}| = \g \subset |C| \).  Since \( C \) spans \( \g \), so does \( C' \).  
	
  For the second part, there exists  \( 0 < R_2 < R \) with  \( \g \subset \O_{R_2} \).
Let \( R_1 > R \).   Suppose there exists \( S_1 \in  {\cal S}_2(\O_{R_1}, \g, q) \) with \( A(S_1) < A(S_0) \).   As before,   there exists \( C_1 =  \sum_{i=1}^s P_{X_i} \widetilde{\t_i}  \in {\cal I}_2(\O_{R_1}, \g, q) \) that spans \( \g \), \( A(C_1) < A(S_0) \), \( \p C_1 = P_Y \widetilde{\g} \), and the \( \{\t_i\} \) are non-overlapping.  Furthermore,    \( |C_1| = \cup \t_i \) has no free edges except those edges whose union is \( \g \).  Let  \( F:\O_{R_1} \to \O_R \) be a contracting diffeomorphism that is the identity on \( \O_{R_2} \).  Then   \( \cup F(\t_i) \) is a union of non-overlapping embedded cells which has no free edges except for those edges whose union is \( \g \). Furthermore, \( \cup F(\t_i) \) spans \( \g \).  It follows that there  exists an integral dipole chain \( C_2 \in {\cal I}_2(\O_R, \g, q)  \) with \( |C_2| = \cup F(\t_i) \), \( C_2 \) spans \( \g \), and \( A(C_2) = {\cal H}^2(\cup F(\t_i)) \le {\cal H}^2(\cup \t_i) = {\cal H}^2(|C_1|) = A(C_1) < A(S_0) \), a contradiction.  
\end{proof}

 The differential $2$-chain \( S_0 \) is therefore a solution to the general problem of Plateau, proving Theorem \ref{thm:Plat}.   
 
\begin{remarks}\label{remarksend}  \mbox{}
  \begin{itemize}
  	\item The extension to codimension \( j \) surfaces with multiple junctions is obtained simply by replacing the vector fields \( X \) and \( Y \) with \( j \)-vector fields.   Figure \ref{fig:tube} illustrates a curve with a junction in \( \R^3 \) and its dipole representative which has dimension $2$ as a chain.    
		    \begin{figure}[htbp]
		     	\centering
		     		\includegraphics[height=1.5in]{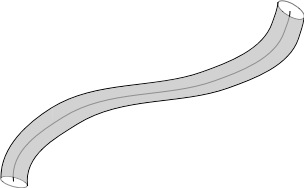}
		     	\caption{A branched curve of codimension two represented by a dipole}
		     	\label{fig:tube}
		     \end{figure}
	   	  	\item A \emph{closed frame} \( \b \) is defined to a union of closed curves.   Any frame \( \cup_{i=1}^s \g_i \) supports an integral dipole curve \( C =    \sum_{i=1}^s P_{X_i} \widetilde{\g_i} \) where \( \g_i \) is smoothly embedded in \( \R^3 \), \( X_i \) is a vector field whose component orthogonal to \( \g_i \) is unit.  The \( X_i \) can be chosen so that \( \p C =  \sum_{i=1}^s \p P_{X_i} \widetilde{\g_i} = 0  \) (see Figure  \ref{fig:branches}).  Any finite number of junctions are permitted, and we still obtain a cycle.
	  We may therefore apply our methods to find a spanning surface of a prescribed closed frame with minimal area (see Figure \ref{Plateausquare}).

 		 \begin{figure}[htbp]
					 	\centering
					 		\includegraphics[height=2in]{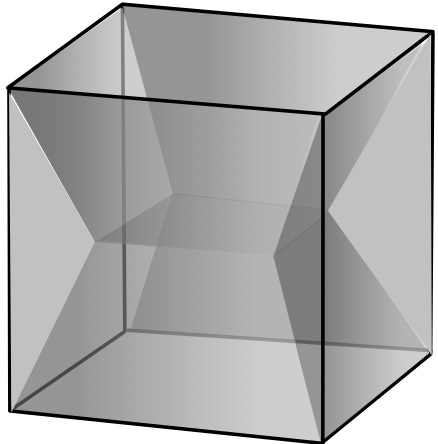}
					 	  \caption{Plateau's square}
					 	\label{Plateausquare}
					 \end{figure}

			  	\item   Smoothly embedded curves that are not cycles can have spanning surfaces as in Figure   \ref{EmptyLoop.jpg}.    The question arises\footnote{Almgren used this example in  \cite{almbulletin} to defend the lack of a boundary operator for varifolds.   He wrote, ``in many  of the phenomena to which our results are applicable there seems no natural	notion of such an operator.''\,|\,}, what is the boundary of such a surface \( S \)?  
					 Our methods apply to this surface,  and produce  \( \p S =  \p E_X \widetilde{\g} \) where \( \g \) is the part of the embedded closed curve that meets \( |S| \). 
					 
		\item	   It is an interesting question to state and pose a version of Plateau's problem for \emph{frames} which are defined as unions of smoothly embedded arcs which are not necessarily closed.  The definition of ``span'' has to be reformulated as a first step.  It would then be desirable to find a condition on an arc to guarantee existence of a nontrivial spanning surface.  
			
				\begin{figure}[htbp]
					\centering
						\includegraphics[height=2in]{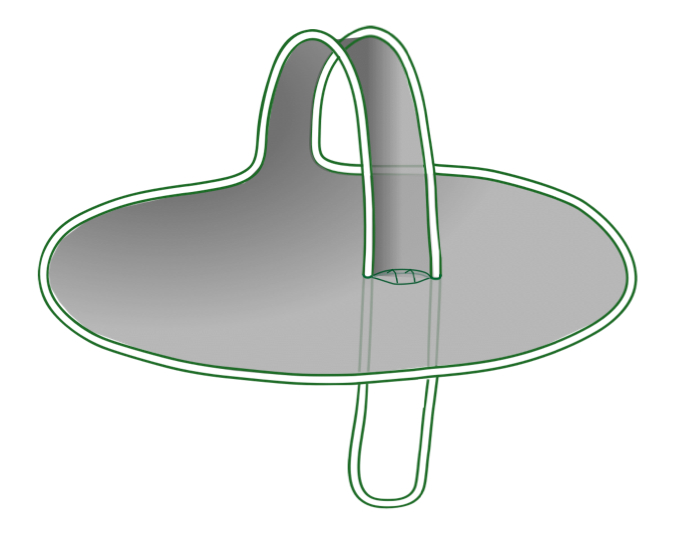}
					\caption{The boundary \( \widetilde{\g} \) of this film is a cycle, but its support \( \g \) is not.}
					\label{EmptyLoop.jpg}
				\end{figure}
				\item 	  Other constraints are possible using the continuous operators on chains available to us, not just boundary. An intriguing example is \( \perp \p \perp \).   These are topics for further research. 
				  \end{itemize}

\end{remarks}

	\addcontentsline{toc}{section}{References}
	\bibliography{Jennybib.bib, mybib.bib}{}

\providecommand{\bysame}{\leavevmode\hbox to3em{\hrulefill}\thinspace}
\providecommand{\MR}{\relax\ifhmode\unskip\space\fi MR }
\providecommand{\MRhref}[2]{%
  \href{http://www.ams.org/mathscinet-getitem?mr=#1}{#2}
}
\providecommand{\href}[2]{#2}
\begin{thebibliography}{Har04b}

\bibitem[Alm66]{varifolds}
Frederick~J. Almgren, \emph{Plateau's problem, an invitation to varifold
  geometry}, Benjamin, 1966.

\bibitem[Alm75]{almbulletin}
\bysame, \emph{Existence and regularity almost everywhere of solutions to
  elliptic variational problems with constraints}, Bulletin of the American
  Mathematical Society \textbf{81} (1975), no.~1, 151--154.

\bibitem[Alt73]{alt}
Hans~Wilhelm Alt, \emph{Verzweigungspunkte von h-fl{\"a}chen. ii.}, Math. Ann.
  \textbf{201} (1973), 33--55.

\bibitem[Dou31]{douglas}
Jesse Douglas, \emph{Solutions of the problem of {Plateau}}, Transactions of
  the American Mathematical Society \textbf{33} (1931), 263--321.

\bibitem[Fed69]{federer}
Herbert Federer, \emph{Geometric measure theory}, Springer, Berlin, 1969.

\bibitem[FF60]{federerfleming}
Herbert Federer and Wendell~H. Fleming, \emph{Normal and integral currents},
  The Annals of Mathematics \textbf{72} (1960), no.~3, 458--520.

\bibitem[Fle62]{flemingregular}
Wendell~H. Fleming, \emph{On the oriented {Plateau} problem}, Rendiconti del
  Circolo Matematico di Palermo \textbf{11} (1962), no.~1, 69--90.

\bibitem[Fle66]{fleming}
\bysame, \emph{Flat chains over a finite coefficient group}, Transactions of
  the American Mathematical Society \textbf{121} (1966), no.~1, 160--186.

\bibitem[Gul73]{gulliver}
Robert Gulliver, \emph{Regularity of minimizing surfaces of prescribed mean
  curvature}, Annals of Mathematics \textbf{97} (1973), no.~2, 275--305.

\bibitem[Har04a]{soap}
Jenny Harrison, \emph{Cartan's magic formula and soap film structures}, Journal
  of Geometric Analysis \textbf{14} (2004), no.~1, 47--61.

\bibitem[Har04b]{plateau}
\bysame, \emph{On {Plateau}'s problem for soap films with a bound on energy},
  Journal of Geometric Analysis \textbf{14} (2004), no.~2, 319--329.

\bibitem[Har10]{OC}
\bysame, \emph{Operator calculus -- the exterior differential complex},
  submitted, December 2010.

\bibitem[Har12]{measuretheory}
\bysame, \emph{Differential chains, measures, and additive set functions}, July
  2012.

\bibitem[HS79]{hardtsimon}
Robert Hardt and Leon Simon, \emph{Boundary regularity and embedded solutions
  for the oriented {Plateau} problem}, Bulletin of the American Mathematical
  Society \textbf{1} (1979), no.~1, 263--265.

\bibitem[Mor88]{morgan}
Frank Morgan, \emph{Geometric measure theory: A beginners guide}, Academic
  Press, London, 1988.

\bibitem[Oss70]{osserman}
Robert Osserman, \emph{A proof of the regularity everywhere of the classical
  solution to {Plateau}'s problem}, Annals of Mathematics \textbf{91} (1970),
  550--569.

\bibitem[Pla73]{plateauoriginal}
Joseph Plateau, \emph{Experimental and theoretical statics of liquids subject
  to molecular forces only}, Gauthier-Villars, 1873.

\bibitem[Rei60]{reifenberg}
E.R. Reifenberg, \emph{Solution of the plateau problem for m-dimensional
  surfaces of varying topological type}, Acta Mathematica \textbf{80} (1960),
  no.~2, 1--14.

\bibitem[Whi57]{whitney}
Hassler Whitney, \emph{Geometric integration theory}, Princeton University
  Press, Princeton, NJ, 1957.

\bibitem[Zie62]{ziemer}
W.P. Ziemer, \emph{Integral currents mod 2}, Transactions of the American
  Mathematical Society \textbf{105} (1962), 496--524.

\bibitem[Zie69]{ziemerplateau}
William~P. Ziemer, \emph{Plateau's problem: an invitation to varifold
  geometry}, Bulletin of the American Mathematical Society (1969), 924--925.

\end{thebibliography}
	\bibliographystyle{amsalpha}

\end{document}